\documentclass{amsart}
\usepackage{graphicx}
\usepackage{amsmath,amssymb}
\usepackage{color, tikz}
\usepackage[utf8]{inputenc}

\def\one{\mbox{1\hspace{-4.25pt}\fontsize{12}{14.4}\selectfont\textrm{1}}}
\newtheorem{thm}{Theorem}[section]
\newtheorem{cor}[thm]{Corollary}
\newtheorem{lem}[thm]{Lemma}
\newtheorem{prop}[thm]{Proposition}

\theoremstyle{definition}
\newtheorem{defn}[thm]{Definition}
\theoremstyle{remark}
\newtheorem{rem}[thm]{Remark}
\theoremstyle{example}

\numberwithin{equation}{section}

\renewcommand{\parallel}{\mathrel{/\mkern-5mu/}}
\makeatletter
\newcommand{\notparallel}{%
  \mathrel{\mathpalette\not@parallel\relax}%
}
\newcommand{\not@parallel}[2]{%
  \ooalign{\reflectbox{$\m@th#1\smallsetminus$}\cr\hfil$\m@th#1\parallel$\cr}%
}
\makeatother

 \theoremstyle{plain}

\newcommand{\norm}[1]{\left\Vert#1\right\Vert}
\newcommand{\abs}[1]{\left\vert#1\right\vert}

\newcommand{\R}{{\mathbb R}}

\newcommand{\T}{{\mathbb T}} 

\newcommand{\Z}{{\mathbb Z}}

\newcommand{\calF}{{\mathcal F}}

\newcommand{\calN}{{\mathcal N}}

\newcommand{\calS}{{\mathcal S}}
\newcommand{\calT}{{\mathcal T}}

\begin{document}

\title[]{Global existence of a non-local semilinear parabolic equation with advection and applications to shear flow}%

\author{Yu Feng, Bingyang Hu, Xiaoqian Xu and Yeyu Zhang}%

\address{Yu Feng: Beijing International Center for Mathematical Research, Peking University, No. 5 Yiheyuan Road Haidian District, Beijing, P.R.China 100871}
\email{fengyu@bicmr.pku.edu.cn}

\address{Bingyang Hu: Department of Mathematics, Purdue University, 150 N. University St., W. Lafayette, IN 47907, U.S.A.}%
\email{hu776@purdue.edu}

\address{Xiaoqian Xu: Zu Chongzhi Center for Mathematics and Computational Sciences, Duke Kunshan University, China}
\email{xiaoqian.xu@dukekunshan.edu.cn}

\address{Yeyu Zhang: School of Mathematics, Shanghai University of Finance and Economics, Shanghai, 200433, PR China}%
\email{zhangyeyu@mail.shufe.edu.cn}

\date{\today}

\subjclass[2010]{}%

\thanks{}

\maketitle

\begin{abstract}
In this paper, we consider the following non-local semi-linear parabolic equation with advection: for $1 \le p<1+\frac{2}{N}$ and $N \ge 1$, 
$$
\begin{cases}
u_t+v \cdot \nabla u-\Delta u=|u|^p-\int_{\T^N} |u|^p \quad & \textrm{on} \quad \T^N, \\
\\
u \ \textrm{periodic} \quad & \textrm{on} \quad  \partial \T^N 
\end{cases}
$$
with initial data $u_0$ defined on $\T^N$. Here $v$ is an incompressible flow, and $\T^N=[0, 1]^N$ is the $N$-torus. 

We first prove local existence of mild solutions for arbitrary data in $L^2$. We then study the global existence of the solutions under the following two scenarios: (1). when $v$ is a mixing flow; (2). when $v$ is a shear flow. More precisely, we show that under these assumptions, there exists a global solution to the above equation in the sense of $L^2$.

\end{abstract}

\tableofcontents


\section{Introduction}
The dissipation effect appears ubiquitously in different areas, such as marine physics, thermodynamics, hydrodynamics, and medical science. Normally, the dissipation effect is known to homogenize the molecular distribution and reduce the system's energy. From the mathematical point of view, the dissipation effect corresponds to the presence of a positive operator in a partial differential equation, for instance, the negative Laplace operator that appears in the heat equation. Correspondingly, the homogenization of molecular and energy dissipation is demonstrated by the solution of the equation converging to a steady state. Many research indicates that presenting the first-order term in such equations may spontaneously enhance the dissipation effect, accelerate energy dumping, and avoid singularity formation appearing under the time evolution. 

Physically, for a time evolution equation that describes the movement of massive small particles, the first-order term can be considered as the transporting of the particles. The mechanism behind the dissipation enhancement effect is: the first-order (advection) term rearranges the particles' distribution to create small-scale structures, which are dumped dramatically by the diffusion.  This phenomenon is closely connected with turbulence, which has been studied extensively in the fluid dynamics community. Among them, the authors of \cite{CKRZ08} first provide a spectrum criterion and examples for dissipation enhancement time-independent first-order term on $\T^d$ for the following standard advection-diffusion equation:
\begin{equation*}
\theta_t+v\cdot\nabla\theta=\Delta \theta.    
\end{equation*}
Besides the spectrum consideration, another way to describe such a phenomenon is to treat such small-scale creations as mixing fluids. The study of mixing in the area of partial differential equations dates back to \cite{LTC11, Thi12}, where the authors consider transport equation
\begin{equation*}
    \eta_t+v\cdot\nabla \eta=0,
\end{equation*}
and quantify the mixing rate of $v$ in terms of the negative Sobolev norms of $\eta$. Such a framework was followed by many researchers \cite{FG19, CDE20, FHX21A}. A significant problem that arises from such a framework is how to construct a smooth and incompressible flow to satisfying the optimal mixing properties, see \cite{ZY17, TZ19} for more attempts in this direction. It is widely believed that some random flows with simple structures such as alternating sine shear flows would also be the optimal mixing flow \cite{Pierrehumbert94}. However, we are not aware of any rigorous study on it. We refer to \cite{BBP19} as a pioneering work in this direction.

Beyond the dissipation enhancement, people are also interested in the interaction between the advection, diffusion, and additional nonlinear terms in the time-evolution equations. It has been proved that the presence of an advection term could prevent the solution from forming a singularity. The earliest work trace back to \cite{KX16}, in which the authors proved that the mixing flows with large amplitude prevent blow-up of the Keller-Segel equation. Recently, the authors in \cite{IXZ21} showed that simply enhance dissipation to a certain degree (quantified in terms of dissipation time) is sufficient to suppress the formation of singularities, see \cite{FFIT19, FM20, FHX21B} for more examples in this direction. For practice, the authors also provide flow with arbitrarily short dissipation time by rescaling a general class of smooth cellular flows, which is not mixing, when the diffusive term is the Laplacian.

Another important class of flows is the shear flow. As a particular solution to the Navier-Stokes equation, discovering the stability of shear flow and the interaction with other nonlinear terms induce a great deal of research. One vital direction of these works is the (partial) dissipation enhancement of the shear flow \cite{BC17, CDE20, WZZ18, WZZ19} and its application \cite{BH17, He18, CDFM21}. Since shear flows can be considered as a mixing flow in one direction, it prevents the formation of singularities of nonlinear partial differential equations from one dimension, see \cite{BH17} as a typical example based on the Keller-Segel equation.  

On the other hand, comparing with the nonlinear time evolution equations such as Keller-Segel equations, another well-known example is the following semilinear parabolic equtaion:
\begin{equation}\label{fujita}
u_t=\triangle u+ u^p, x\in \mathbb{R}^N,
\end{equation}
here usually we consider the nonnegative solution to such equation. The study of this equation can track back to \cite{Fujita66}. Considering the global existence or finite time blow up to such equation, there is a well-known Fujita critical exponent $p_F=1+\frac{2}{N}$, for which it can be shown that (see, \cite{Fujita66,Levine90}):
\begin{enumerate}
\item [(1).]  $p_F\geq p>1$, there is no nontrivial nonnegative solution for this equation;

\medskip
\item [(2).] $p>p_F$, there exists positive solution to this problem with initial data to be sufficiently small.
\end{enumerate}
In this paper, we will consider a similar equation, with additional non-local term, on the periodic setting, and ignoring the positivity requirement, as follows:
\begin{equation} \label{withoutv}
\begin{cases}
u_t-\Delta u=|u|^p-\int_{\T^N} |u|^p \quad & \textrm{on} \quad \T^N, \\
\\
u \ \textrm{periodic} \quad & \textrm{on} \quad  \partial \T^N,
\end{cases}
\end{equation}
The problem \eqref{withoutv} was studied extensively by \cite{JK08, SJM07}, in which, the authors considered the local existence result with a continuous initial data up to boundary. Moreover, they also showed the global existence to such equations with small $L^\infty$ initial data. Equation \eqref{withoutv} is also relevant to the Navier-Stokes equations on an infinite slab (see, e.g., \cite{CWS94}), one may expect that to understand the singularity formation or global existence of the solution to such equation can help us understand the behavior of Navier-Stokes equation.\\

Following the idea of using advection term to prevent blow-up, we now consider the following non-local semi-linear parabolic equation with advection on torus $\T^N:=[0, 1]^N$ for some $N \ge 2$: for $1 \le p<1+\frac{2}{N}$, 
\begin{equation} \label{maineq}
\begin{cases}
u_t+v \cdot \nabla u-\Delta u=|u|^p-\int_{\T^N} |u|^p \quad & \textrm{on} \quad \T^N, \\
\\
u \ \textrm{periodic} \quad & \textrm{on} \quad  \partial \T^N,
\end{cases}
\end{equation}
with initial condition $u(x, 0)=u_0(x)$ on $\T^N$.  Here $v$ is an incompressible flow. The reason for us to restrict our attention on the case when $1 \le p<1+\frac{2}{N}$ will be clear from the explicit computation later. Heuristically, different from the more complicated non-negative solution case \eqref{fujita}, note that when $p$ is larger, the right hand side of \eqref{maineq} might have many ``peaks", which can possibly prevent the diffusion, this in turns suggests that we have more chance to derive a globally existed solution when $p$ is small.

The goal of this paper is to study the global existence of \eqref{maineq}. In particular, our main results can be stated as follows:
\begin{thm}\label{thm11}
Let $N=2$ or $3$. For $1\leq p<1+\frac{2}{N}$, there exists divergence free $v(t,x)\in L^{\infty}([0,\infty),L^{\infty})$, depending on $p$, $u_0$, such that the mild solution (see Definition \ref{mildsol}) to \eqref{maineq} exists globally in time. 
\end{thm}
\begin{thm}\label{thm12}
Let $N=2$. For $1\leq p<2$, if
\begin{equation*}
  \int_{\mathbb{T}} \left(\int_{\mathbb{T}}u_0(x_1,x_2)dx_1 \right)^2dx_2\ll 1,
\end{equation*}
then there exists shear flows ${\bf v}=(v_1(x_2), 0) \in L^\infty$, depending only $p$ and $u_0$ such that the mild solution to \eqref{maineq} exists globally in time.
\end{thm}

\begin{rem}
 In the sequel, we will study both Theorem \ref{thm11} and Theorem \ref{thm12} in a more precise and quantitative way, see Section \ref{Sec3} and Section \ref{Sec4}, respectively. 
\end{rem}

We start with some basic setup. Given a function $f \in L^p(\T^N)$, we denote $\hat{f}({\bf k})$ to be its Fourier coefficient of $f$ at frequency ${\bf k} \in \Z^N$. Then the \emph{inhomogeneous Sobolev space} $H^s(\T^N), s \in \R$ is defined to be the collection of all measurable functions $f$ on $\T^N$, with
$$
\|f\|_{H^s}^2:=\sum_{{\bf k} \in \Z^N} (1+|{\bf k}|^2)^s \left| \hat{f}({\bf k}) \right|^2=\left\| \left(I-\Delta \right)^{s/2} f \right\|_2^2<\infty, 
$$
while the \emph{homogeneous Sobelev space} $\dot{H}^s(\T^N), s \in \R$ consists of all measurable functions $f$ on $\T^N$ with 
$$
\|f\|_{\dot{H}^s}^2:=\sum_{{\bf k} \in \Z^N\setminus\{{\bf 0}\}} |{\bf k}|^{2s} \left| \hat{f}({\bf k}) \right|^2=\left\| \left(-\Delta \right)^{s/2} f \right\|_2^2<\infty.
$$
Note that it is clear that for $f \in L^2(\T^N)$, $f \in H^s(\T^N)$ if and only if $f \in \dot{H}^s(\T^N)$.

Let $e^{t\Delta}$ be the semigroup generated by the Laplacian $\Delta$, namely, 
$$
e^{t\Delta}f= \calF^{-1} \left(e^{-t\left|{\bf k} \right|^2} \hat{f} \right),
$$
where $\calF^{-1}$ is the inverse Fourier transform on $\Z^N$.

Here are some basic properties of the  semigroup $e^{t\Delta}$, whose proof are straightforward from the definitions. 

\begin{lem} \label{lem01}
The following estimates hold:
\begin{enumerate}
    \item [(1).] $\left\| (-\Delta)^{\frac{s}{2}} e^{t\Delta} f \right\|_{L^2} \le C_st^{-\frac{s}{2}} \|f\|_{L^2}$;
    \medskip
    \item [(2).] For any $1 \le p \le 2$, 
    \begin{equation} \label{eq2020}
    \left\|e^{t\Delta} f\right\|_{L^2} \lesssim t^{-\frac{N(p-1)}{4}} \|f\|_{L^{\frac{2}{p}}}. 
    \end{equation}
\end{enumerate}
\end{lem}

\begin{proof}
The first assertion is straightforward by analysing the Fourier side, and hence we omit it here. While the second assertion follows from interpolation. More precisely, by H\"older's inequality with the conjugate pair $\left(\frac{1}{2-p}, \frac{1}{p-1} \right)$, we have
\begin{eqnarray*}
\|e^{t\Delta} f\|_{L^2}^2%
&=& \sum_{{\bf k} \in \Z^N} e^{-2t|{\bf k}|^2} \left| \hat{f} \left({\bf k} \right) \right|^2 \\ 
&\le& \left(\sum_{{\bf k} \in \Z^N} e^{\frac{-2t|{\bf k}|^2}{p-1}} \right)^{p-1} \cdot \left( \sum_{{\bf k} \in \Z^N} \left| \hat{f}({\bf k}) \right|^{\frac{2}{2-p}} \right)^{2-p} \\
&\lesssim& \|f\|_{L^{\frac{2}{p}}}^2 \cdot \left(\int_{\R^N} e^{-\frac{2t|x|^2}{p-1}} dx \right)^{p-1} \lesssim t^{-\frac{N(p-1)}{2}} \cdot \|f\|_{L^{\frac{2}{p}}}^2, 
\end{eqnarray*}
which clearly implies the desired estimate \eqref{eq2020}. Here in the above estimate, we have used the Hausdoff-Young's inequality on $\T^N$ (or interpolation) 
$$
\|\hat{g} \|_{\ell^q (\Z^N)} \lesssim \|g\|_{L^{q'}(\T^N)}, \quad  q \ge 2,
$$
where $q=\frac{2}{2-p}$ in our case. 
\end{proof}

In this paper, we will mainly consider the mild solution to \eqref{maineq}, which defines as follows. 

\begin{defn} \label{mildsol}
Let $1 \le p<1+\frac{2}{N}$, $v(t,x)\in L^{\infty}([0,\infty);L^2(\mathbb{T}^N))$. A function $u \in C\left([0, T]; L^2(\T^N) \right)$, $T>0$ with $\nabla u(t) \in L^2(\T^N)$ for each $0<t\leq T$, is called a \emph{mild solution} of \eqref{maineq} with initial data $u_0 \in L^2(\T^N)$, if it satisfies 
\begin{eqnarray*}
u(t)=\calN(u)(t)%
&:=& e^{t \Delta} u_0+\int_0^t e^{(t-\tau)\Delta} \left(|u|^p \right)d\tau -\int_0^t e^{(t-\tau)\Delta} \left(v \cdot \nabla u \right)d\tau \\
&& -\int_0^t \int_{\T^N} |u|^p dxd\tau,
\end{eqnarray*}
where the integrals above are in the sense of B\"ochner integral. 
\end{defn}

In addition, the definition of the weak solution of \eqref{maineq} is as follows.

\begin{defn}
Let $1 \le p<1+\frac{2}{N}$, $v(t,x)\in L^{\infty}([0,\infty);L^2(\mathbb{T}^N))$. A function $u \in L^\infty \left([0, T]; L^2\left(\T^N \right) \right) \cap L^2\left([0, T]; H^1(\T^N) \right)$ is called a \emph{weak solution} of \eqref{maineq} on $[0, T)$ with initial value $u(0)=u_0 \in L^2(\T^N)$ if, for all $\varphi \in C_c^\infty\left([0, T) \times \T^N \right)$, 
\begin{eqnarray} \label{weaksol}
&& \int_{\T^N} u_0\varphi(0)dx+\int_0^T \int_{\T^N} u \partial_t \varphi dxdt = \int_0^T \int_{\T^N} \varphi \left(v \cdot \nabla u\right) dxdt \nonumber \\
&&\quad \quad \quad \quad -\int_0^T \int_{\T^N} |u|^p\varphi dxdt   +\int_0^T \int_{\T^N} \nabla u \nabla \varphi dxdt \nonumber \\
&&\quad \quad \quad \quad +\int_0^T \left(\int_{\T^N} |u|^pdx \right) \left(\int_{\T^N} \varphi dx \right) dt.
\end{eqnarray}
and $\partial_t u \in L^2\left(\left[0, T\right]; H^{-1}(\T^N) \right)$. 
\end{defn}

\section{Short-time existence of the mild solution with data in $L^2$}\label{sec2}

\begin{defn}
For any $0<T<1$, we let the Banach space $X_T$ to be 
$$
X_T:=C\left([0, T]; L^2(\T^N) \right) \cap \left\{u: \T^N \times \R_+ \to \R \ \big|  \ \sup_{0<t \le T} t^{\frac{1}{2}} \|\nabla u\|_{L^2}<\infty \right\},
$$
with the norm 
$$
\|u\|_{X_T}:=\max \left( \sup_{0 \le t \le T} \|u\|_{L^2}, \ \sup_{0<t \le T} t^{\frac{1}{2}} \|\nabla u \|_{L^2} \right). 
$$
\end{defn}

We have the following result. 

\begin{thm} \label{mainthm01}
Let $1 \le p<1+\frac{2}{N}$, $u_0 \in L^2(\T^N)$ and let $v \in L^\infty \left(\R_+; L^\frac{2}{p-1} \right)$. There exists  $0<T \le 1$ depending on $\sup_{t>0} \|v\|_{L^\frac{2}{p-1}}$ and $\|u_0\|_{L^2}$ such that \eqref{maineq} admits a mild solution $u$ on $[0, T]$, which is unique in $X_T$.
\end{thm}

We divide the proof of Theorem \ref{mainthm01} into several parts. 

\begin{lem} \label{lem012}
Under the assumption of Theorem \ref{mainthm01}, we have for any $0<T \le 1$
$$
\calN(u) \in C\left([0, T]; L^2(\T^N) \right). 
$$
Moreover, for each $t \in (0, T]$, there exists some $C_1=C_1(N, p)>0$, such that
\begin{equation} \label{eq100}
\|\calN(u)(t)\|_{L^2} \le C_1 \left(\|u_0\|_{L^2}+t^{1-\frac{N(p-1)}{4}} \|u\|_{X_T}^p+  t^{\frac{2-N(p-1)}{4}} \|v\|_{L^\infty\left(\R_+; L^\frac{2}{p-1}\right)} \|u\|_{X_T}  \right). 
\end{equation} 
\end{lem}

\begin{proof}
The first claim follows clearly from the basic properties of semigroups and B\"ocher integrals, and hence we only focus on proving the estimate \eqref{eq100}. For any $t \in (0, 1]$, using Lemma \ref{lem01}, we have 
\begin{eqnarray} \label{fengyuxiaomeiren}
\|\calN(u)(t)\|_{L^2}%
&\lesssim& \|e^{t\Delta}u_0\|_{L^2}+ \int_0^t \int_{\T^N} |u|^pdx d\tau + \int_0^t \left\|e^{(t-\tau)\Delta}\left(|u(\tau)|^p \right) \right\|_{L^2} d\tau \nonumber \\
&& \quad \quad \quad + \int_0^t \left\|e^{(t-\tau)\Delta}\left(v(\tau) \cdot \nabla u(\tau) \right) \right\|_{L^2} d\tau \nonumber \\
&\lesssim& \|u_0\|_{L^2}+t\|u\|_{X_T}^p+\int_0^t (t-\tau)^{-\frac{N(p-1)}{4}} \left\||u(\tau)|^p \right\|_{L^\frac{2}{p}} d\tau  \nonumber \\
&&\quad \quad \quad  +\int_0^t (t-\tau)^{-\frac{N(p-1)}{4}} \|v(\tau) \cdot \nabla u(\tau)\|_{L^{\frac{2}{p}}} d\tau \nonumber \\
&=& \|u_0\|_{L^2}+t\|u\|_{X_T}^p+\int_0^t (t-\tau)^{-\frac{N(p-1)}{4}} \left\|u(\tau) \right\|_{L^2}^p d\tau  \nonumber \\
&&\quad \quad \quad  +\int_0^t (t-\tau)^{-\frac{N(p-1)}{4}} \|v(\tau) \cdot \nabla u(\tau)\|_{L^\frac{2}{p}} d\tau. 
\end{eqnarray} 
For the first integral in \eqref{fengyuxiaomeiren}, using the fact that $1 \le p<1+\frac{2}{N}$, we see that we can bound it by
\begin{equation} \label{xuanxuan}
\|u\|_{X_T}^p  \cdot \int_0^t (t-\tau)^{-\frac{N(p-1)}{4}} d\tau \lesssim t^{1-\frac{N(p-1)}{4}} \|u\|_{X_T}^p. 
\end{equation} 
While for the second integral \eqref{fengyuxiaomeiren}, we can bound it above by 
\begin{eqnarray*}
&&\int_0^t (t-\tau)^{-\frac{N(p-1)}{4}} \|v(\tau)\|_{L^\frac{2p'}{p}} \|\nabla u(\tau)\|_{L^2} d\tau \\
&& = \int_0^t (t-\tau)^{-\frac{N(p-1)}{4}} \tau^{-\frac{1}{2}} \|v(\tau)\|_{L^{\frac{2}{p-1}}} \left(\tau^{\frac{1}{2}}\|\nabla u(\tau)\|_{L^2} \right) d\tau \\
&& \le \|v\|_{L^\infty\left(\R_+; L^\frac{2}{p-1}\right)} \cdot \|u\|_{X_T} \cdot \int_0^t (t-\tau)^{-\frac{N(p-1)}{4}} \tau^{-\frac{1}{2}} d\tau \\
&& \lesssim t^{\frac{2-N(p-1)}{4}} \cdot  \|v\|_{L^\infty\left(\R_+; L^\frac{2}{p-1} \right)} \cdot \|u\|_{X_T}.
\end{eqnarray*}
Note that $\frac{2-N(p-1)}{4}>0$, which is guaranteed by the assumption $p<1+\frac{2}{N}$. Combining the above estimates with \eqref{xuanxuan} and \eqref{fengyuxiaomeiren}, \eqref{eq100} clearly holds. 
\end{proof}

\begin{lem} \label{lem02}
Under the assumption of Lemma \ref{lem012}, we have for each $t \in (0, T]$, there exists some $C_2=C_2(N, p)>0$, such that
\begin{equation} \label{eq200}
t^{\frac{1}{2}} \left\|\nabla \calN(u)(t) \right\|_{L^2}  \le C_2 \left(\|u_0\|_{L^2}+t^{1-\frac{N(p-1)}{4}} \|u\|_{X_T}^p+  t^{\frac{2-N(p-1)}{4}} \|v\|_{L^\infty\left(\R_+; L^\frac{2}{p-1}\right)} \|u\|_{X_T}  \right). 
\end{equation} 
\end{lem}

\begin{proof}
For any $t \in (0, 1]$, 
\begin{eqnarray} \label{xuanku}
\left\|\nabla \calN(u)(t) \right\|_{L^2} %
&\lesssim& \left\|\nabla e^{t\Delta} u_0\right\|_{L^2}+\int_0^t \left\| \nabla e^{(t-\tau)\Delta} (|u|^p) \right\|_{L^2}d \tau \nonumber \\
&& +\int_0^t \left\| \nabla e^{(t-\tau)\Delta}(v \cdot \nabla u) \right\|_{L^2} d\tau \nonumber \\
&\lesssim& t^{-\frac{1}{2}} \|u_0\|_{L^2}+\int_0^t \left\| \nabla e^{(t-\tau)\Delta} (|u|^p) \right\|_{L^2}d \tau \nonumber \\
&& +\int_0^t \left\| \nabla e^{(t-\tau)\Delta}(v \cdot \nabla u) \right\|_{L^2} d\tau.
\end{eqnarray}
We estimate the two integrals in \eqref{xuanku} as follows. For the first term, 
\begin{eqnarray} \label{xuan01}
&& \int_0^t \left\| \nabla e^{(t-\tau)\Delta} (|u|^p) \right\|_{L^2}d \tau \nonumber \\
&& \le  \int_0^t \left\|\nabla e^{\frac{(t-\tau)\Delta}{2}} \right\|_{L^2 \to L^2} \left\| e^{\frac{(t-\tau)\Delta}{2}} \left(|u|^p \right) \right\|_{L^2} d\tau \nonumber \\
&& \lesssim \int_0^t (t-\tau)^{-\frac{1}{2}} (t-\tau)^{-\frac{N(p-1)}{4}} \||u|^p\|_{L^\frac{2}{p}}d\tau  \nonumber\\
&& = \int_0^t (t-\tau)^{-\frac{1}{2}} (t-\tau)^{-\frac{N(p-1)}{4}} \|u(\tau)\|_{L^2}^p d\tau  \nonumber\\
&& \le \|u\|_{X_T}^p \cdot \int_0^t (t-\tau)^{-\frac{2+N(p-1)}{4}} d\tau  \lesssim t^{\frac{2-N(p-1)}{4}} \cdot  \|u\|_{X_T}^p 
\end{eqnarray}
Now for the second integral, we have
\begin{eqnarray} \label{xuan02}
&&\int_0^t \left\| \nabla e^{(t-\tau)\Delta}(v \cdot \nabla u) \right\|_{L^2} d\tau \nonumber \\
&& \le  \int_0^t \left\|\nabla e^{\frac{(t-\tau)\Delta}{2}} \right\|_{L^2 \to L^2} \left\| e^{\frac{(t-\tau)\Delta}{2}} \left(v(\tau) \cdot \nabla u(\tau) \right) \right\|_{L^2} d\tau \nonumber\\
&& \lesssim  \int_0^t (t-\tau)^{-\frac{1}{2}} (t-\tau)^{-\frac{N(p-1)}{4}} \|v(\tau) \cdot \nabla u(\tau)\|_{L^\frac{2}{p}} d\tau  \nonumber\\
&& \lesssim  \int_0^t (t-\tau)^{-\frac{2+N(p-1)}{4}} \|v(\tau)\|_{L^{\frac{2}{p-1}}} \|\nabla u(\tau) \|_{L^2} d\tau \nonumber \\
&& \le \|v\|_{L^\infty\left(\R_+; L^{\frac{2}{p-1}} \right)} \cdot \int_0^t (t-\tau)^{-\frac{2+N(p-1)}{4}} \tau^{-\frac{1}{2}} \cdot \left(\tau^{\frac{1}{2}} \|\nabla u(\tau) \|_{L^2} \right) d\tau \nonumber \\
&& \le \|v\|_{L^\infty\left(\R_+; L^{\frac{2}{p-1}} \right)} \cdot \|u\|_{X_T} \cdot \int_0^t (t-\tau)^{-\frac{2+N(p-1)}{4}} \tau^{-\frac{1}{2}} d\tau \nonumber \\
&& \lesssim t^{-\frac{N(p-1)}{4}} \cdot \|v\|_{L^\infty\left(\R_+; L^{\frac{2}{p-1}} \right)} \cdot \|u\|_{X_T}. 
\end{eqnarray}
Then desired estimate \eqref{eq200} follows from \eqref{xuanku}, \eqref{xuan01}, \eqref{xuan02} and the identities
$$
\frac{2-N(p-1)}{4}+\frac{1}{2}=1-\frac{N(p-1)}{4}
$$
and
$$
-\frac{N(p-1)}{4}+\frac{1}{2}=\frac{2-N(p-1)}{4}.
$$
\end{proof}

Combining both Lemma \ref{lem012} and Lemma \ref{lem02}, we have the following result. 
\begin{lem} \label{lem101}
For any $0<T \le 1$. The map $\calN: X_T \to X_T$ and there exists some constant $C=C(N, p)>0$, such that
$$
\|\calN(u)\|_{X_T} \le  C \left(\|u_0\|_{L^2}+T^{1-\frac{N(p-1)}{4}} \|u\|_{X_T}^p+  T^{\frac{2-N(p-1)}{4}} \|v\|_{L^\infty\left(\R_+; L^\frac{2}{p-1}\right)} \|u\|_{X_T}  \right). 
$$
\end{lem}

\medskip

Our next result shows that the map $\calN$ is Lipschitz on $X_T$. 

\begin{lem} \label{lem102}
Let $0<T \le 1$. Then exists a constant $\widetilde{C}=\widetilde{C}(N, p)>0$ such that for $u_1, u_2 \in X_T$,
\begin{eqnarray*} 
&&\left\|\calN(u_1)-\calN(u_2) \right\|_{X_T}  \nonumber \\ 
&& \quad \le \widetilde{C} \left[T^{1-\frac{N(p-1)}{4}} \left(\|u_1\|^{p-1}_{X_T}+\|u_2\|^{p-1}_{X_T}\right)+ T^{\frac{2-N(p-1)}{4}}   \|v\|_{L^\infty\left(\R_+; L^{\frac{2}{p-1}}\right)}\right]\|u_1-u_2\|_{X_T}.
\end{eqnarray*}
\end{lem}

\begin{proof}
The proof of Lemma \ref{lem102} is similar to the one of Lemma \ref{lem101}. Here we only include the necessary modification and we would like to leave the details to the interested reader. Clearly, the desired estimate follows from the following two estimates
\begin{eqnarray*}
&&\left\|\calN(u_1)(t)-\calN(u_2)(t) \right\|_{L^2}\\ 
&&\quad \lesssim_{N, p} \left[t^{1-\frac{N(p-1)}{4}} \left(\|u_1\|^{p-1}_{X_T}+\|u_2\|^{p-1}_{X_T}\right)+t^{\frac{2-N(p-1)}{4}}   \|v\|_{L^\infty\left(\R_+; L^{\frac{2}{p-1}}\right)} \right]\|u_1-u_2\|_{X_T}
\end{eqnarray*}
and
\begin{eqnarray*}
&&\|\nabla \calN(u_1)(t)-\nabla \calN(u_2)(t)\|_{L^2}\\
&&\quad \lesssim_{N, p} \left[t^{\frac{2-N(p-1)}{4}} \left(\|u_1\|^{p-1}_{X_T}+\|u_2\|^{p-1}_{X_T}\right)+t^{\frac{-N(p-1)}{4}}   \|v\|_{L^\infty\left(\R_+; L^{\frac{2}{p-1}}\right)} \right]\|u_1-u_2\|_{X_T}.
\end{eqnarray*}

These estimates follows from the proof of Lemma \ref{lem012} and Lemma \ref{lem02}, respectively, and the elementary estimate
$$
\left|a^p-b^p\right| \lesssim_p (a+b)^{p-1}|a-b|, \quad a,b \ge 0. 
$$
\end{proof}

We are ready to prove Theorem \ref{mainthm01}. 

\begin{proof} [Proof of Theorem \ref{mainthm01}.]
Let $\mathfrak C:=\max\{C, \widetilde{C}\}$, where $C$ and $\widetilde{C}$ are defined in Lemma \ref{lem101} and Lemma \ref{lem102}, respectively. Let $M:=10\mathfrak C \|u_0\|_{L^2}$ and we let $B(0, M)$ be the ball in $X_T$ centered at origin with radius $M$. Further, let
$$
0<T \le \min \left\{1, \frac{1}{\left(10 \mathfrak CM^{p-1} \right)^{\frac{4}{4-N(p-1)}}+\left(10 \mathfrak C \|v\|_{L^\infty \left(\R_+; L^{\frac{2}{p-1}}\right)} \right)^{\frac{4}{2-N(p-1)}}}\right\}.
$$
It is then easy to check that
$$
\left\|\calN(u)\right\|_{X_T} \le M, \quad \forall u \in B(0, M)
$$
and
$$
\left\|\calN(u_1)-\calN(u_2) \right\|_{X_T} \le \frac{1}{2} \|u_1-u_2\|_{X_T}, \quad u_1, u_2 \in B(0, M). 
$$
Theorem \ref{mainthm01} then follows clearly from the Banach fixed point theorem. 
\end{proof}

\begin{cor} \label{cor01}
Under the assumption of Theorem \ref{mainthm01}, if $T^*$ is the maximal time of existence of the mild solution $u=\calN(u)$, then 
$$
\limsup_{t \to T_{-}^*} \|u(t)\|_{L^2}=\infty.
$$
Otherwise, $T^*=\infty$. 
\end{cor}

\begin{proof}
Let $T^*<\infty$ and $\limsup\limits_{t \to T_{-}^*} \|u(t)\|_{L^2}=A<\infty$. Therefore, we can find a $t_0<T^*$, such that for all $t'$, $t_0<t'<T^*$, one has $\|u(t')\|_{L^2} \le 2A$. Let
$$
0<\calT<\min \left\{1, \frac{1}{\left(5 \cdot 2^p \cdot \mathfrak CA^{p-1} \right)^{\frac{4}{4-N(p-1)}}+\left(10 \mathfrak C \|v\|_{L^\infty \left(\R_+; L^{\frac{2}{p-1}} \right)} \right)^{\frac{4}{2-N(p-1)}}}\right\}. 
$$
Finally, we fix $t' \in (t_0, T^*)$ satisfying $\calT>T^*-t'$ (see, e.g., Figure \ref{Figure1}). 

\bigskip

\begin{figure}[ht]
\begin{tikzpicture}[scale=4.5]
\draw (0, -0.02) node [below] {$0$}; 
\fill (0, 0) circle [radius=0.01]; 
\draw (0, 0) -- (2.3, 0); 
\draw (2, -0.02) node [below] {$T^*$}; 
\fill (2, 0) circle [radius=0.01]; 
\draw (0.5, -0.02) node [below] {$1$}; 
\fill (0.5, 0) circle [radius=0.01]; 
\draw (0.4, -0.02) node [below] {$\calT$}; 
\fill (0.4, 0) circle [radius=0.01]; 
\draw (1.7, -0.02) node [below] {$t_0$}; 
\fill (1.7, 0) circle [radius=0.01]; 
\draw (1.85, -0.02) node [below] {$t'$}; 
\fill (1.85, 0) circle [radius=0.01]; 
\draw (2.25, -0.02) node [below] {$t'+\calT$}; 
\fill (2.25, 0) circle [radius=0.01]; 
\end{tikzpicture}
\caption{$\calT, t_0, t', T^*$ and $t'+\calT$.}
\label{Figure1}
\end{figure}

\bigskip

Now by Theorem \ref{mainthm01} with initial data $u(t')$, there exists a mild solution $\widetilde{u}$ on $[t', t'+\calT]$. Moreover, the uniqueness of the mild solutions suggests that $u=\widetilde{u}$ on $[t', T^*)$, which suggests that the mild solution can be extended by surpassing $T^*$. This is a contradiction. 
\end{proof}

\begin{rem}
The assumption of Corollary \ref{cor01} is not trivial and we refer the reader Theorem \ref{L2blowup} for an example where the $L^2$ norm of the solution exists only in finite time.
\end{rem}

\medskip

Finally, we show that the mild solution $u=\calN(u)$ constructed on $[0, T)$ coincides the weak solution, under a slightly stronger assumption. 

\begin{thm} \label{mainthm04}
Let $v \in L^\infty \left(\R_+; L^{\infty} \right)$ and $u=\calN(u)$ be the mild solution on $[0, T]$ given in Theorem \ref{mainthm01} with the initial data $u_0 \in L^2(\T^N)$ with \begin{equation} \label{20210814eq01}
\int_{\T^N} u_0 dx=0. 
\end{equation} 
Then, $u=\calN(u)$ is a weak solution of \eqref{maineq}, and satisfies the energy identity for any $0 \le t<T$:
\begin{equation} \label{energyid}
\|u_0\|_{L^2}^2+2\int_0^t\int_{\T^N} |u|^pu dxdt=\|u(t)\|_{L^2}^2+2\int_0^t \int_{\T^N} \left|\nabla u \right|^2dxdt. 
\end{equation}
\end{thm}

\begin{rem}
Note that the assumption of Theorem \ref{mainthm04} can be relaxed. More precisely, one can actually prove Theorem \ref{mainthm04} without assuming \eqref{20210814eq01} and under less regularity assumption: $v \in L^\infty \left(\R_+; L^q \right)$ with $q>\max\left\{N+2, \frac{2}{p-1} \right\}$ (see, Theorem \ref{20210814thm01}). Nevertheless, the current Theorem \ref{mainthm04} has already contained the most interesting cases and suffices for our applications (see, Sections \ref{Sec3} and \ref{Sec4}). 
\end{rem}

\begin{proof}
We divide the proof into several steps.

\medskip

\textit{Step I: An alternative representation of $u=\calN(u)$.} 

\medskip

We claim that for any $0<\epsilon \le t \le T$, one has
\begin{equation} \label{20210426eq01}
u(t)=e^{\left(t-\epsilon \right)\Delta}u(\epsilon)+\int_\epsilon^t e^{(t-\tau)\Delta} \left(|u|^p-v \cdot \nabla u \right)d\tau-\int_\epsilon^t \int_{\T^N} |u|^p dx d\tau,
\end{equation} 
where the above identity is taken in the sense of functions in $C\left(\left[\epsilon, T \right]; L^2(\T^N) \right)$. Indeed, by Definition \ref{mildsol}, we have
\begin{eqnarray} \label{20210426eq02}
u(t)-u(\epsilon)%
&=& \left(e^{(t-\epsilon)\Delta}-I \right) \left[e^{\epsilon \Delta}u_0+\int_0^\epsilon e^{(\epsilon-\tau)\Delta} \left(|u|^p-v \cdot \nabla u \right) d\tau \right] \nonumber \\
&&\quad \quad +\int_\epsilon^t e^{(t-\tau)\Delta} \left(|u|^p-v \cdot \nabla u \right)d\tau-\int_\epsilon^t \int_{\T^N} |u|^pdxd\tau.
\end{eqnarray}
Note that
$$
\left(e^{(t-\epsilon)\Delta}-I\right) \left[ \int_0^\epsilon \int_{\T^N} |u|^pdxd\tau \right]=0.
$$
This together \eqref{20210426eq02} gives the desired claim \eqref{20210426eq01}. 

\medskip

\textit{Step II: For each $0<\epsilon \le t$, $u \in L^2 \left(\left[\epsilon, t\right]; H^1(\T^N) \right)$.}

\medskip

Indeed, using the fact that $u \in X_T$, we have for any $\epsilon \le \tau \le t \le T$, 
$$
\|u(\tau)\|_{H^1(\T^N)} \lesssim \tau^{-1/2} \cdot \|u\|_{X_T},
$$
this gives
$$
\int_\epsilon^t \|u(\tau)\|_{H^1(\T^N)}^2 d\tau \lesssim \left(\log T -\log \epsilon \right) \cdot \|u\|_{X_T}^2.
$$ 
Note that although the above bound depends on $\epsilon$, we will finally show that it can be bounded above uniformly in $\epsilon>0$, and therefore we are able to apply the dominated convergence theorem to pass $\epsilon$ to $0$. We will come back to this point later.

\medskip

\textit{Step III: $L^2$ pairing.}

\medskip

Let $\varphi \in C^\infty \left(\left[\epsilon, T \right] \times \T^N \right)$ and consider the $L^2$ pairing of $\varphi$ with $u$ by using \eqref{20210426eq01}. This gives
\begin{eqnarray*}
 \int_{\T^N} \varphi(t)u(t)dx%
&=&\int_{\T^N} e^{(t-\epsilon)\Delta}\varphi(t)u(\epsilon)dx \\
&& +\int_\epsilon^t \int_{\T^N} e^{(t-\tau)\Delta}\varphi(t) \left(|u(\tau)|^p-v(\tau) \cdot \nabla u(\tau) \right)dxd\tau \\
&&- \int_{\T^N} \varphi(t)dx \cdot \int_\epsilon^t \int_{\T^N} |u(\tau)|^p dxd\tau, 
\end{eqnarray*}
where we have used Fubini-Tonelli's Theorem to exchange the order of the order of integration. Now using the fact that $e^{t\Delta}$ is a strongly continuous semigroup on $L^2(\T^N)$ and standard spectral theory, we have
\begin{eqnarray*}
&&\frac{d}{dt} \int_{\T^N} \varphi(t)u(t)dx = \int_{\T^N} \left(e^{(t-\epsilon)\Delta}\partial_t \varphi(t)+\Delta e^{(t-\epsilon)\Delta}\varphi(t) \right) u(\epsilon)dx \\
&& \quad \quad +\int_{\T^N}\int_\epsilon^t \left(e^{(t-\epsilon)\Delta}\partial_t \varphi(t)+\Delta e^{(t-\epsilon)\Delta}\varphi(t) \right) \left(|u(\tau)|^p-v(\tau) \cdot \nabla u(\tau) \right)dxd\tau \\
&& \quad \quad +\int_{\T^N} \varphi(t) \left(|u(t)|^p-v(t) \cdot \nabla u(t) \right)dx-\int_{\T^N} |u(t)|^pdx \cdot \int_{\T^N} \varphi(t)dx \\
&& \quad \quad  -\int_{\T^N} \partial_t \varphi(t)dx \cdot \int_\epsilon^t \int_{\T^N} |u(\tau)|^pdxd\tau.
\end{eqnarray*}
Using the self-adjointness of $e^{t\Delta}$ and the fact that $\int_{\T^N} \Delta \varphi(t)dx=0$, 
we find that
\begin{eqnarray*}
&&\frac{d}{dt} \int_{\T^N} \varphi(t)u(t)dx \\
&&=\int_{\T^N} \partial_t \varphi(t) \bigg[e^{(t-\epsilon)\Delta} u(\epsilon)+\int_\epsilon^t e^{(t-\epsilon)\Delta} \left(|u(\tau)|^p-v(\tau) \cdot \nabla u(\tau) \right) d\tau \\
&& \quad \quad \quad \quad \quad \quad \quad \quad \quad  \quad \quad  \quad \quad \quad \quad-\int_\epsilon^t\int_{\T^N} |u(\tau)|^pdxd\tau \bigg] dx \\
&& \ \quad  +\int_{\T^N} \Delta \varphi(t) \bigg[e^{(t-\epsilon)\Delta} u(\epsilon)+\int_\epsilon^t e^{(t-\epsilon)\Delta} \left(|u(\tau)|^p-v(\tau) \cdot \nabla u(\tau) \right) d\tau  \\
&& \quad \quad \quad \quad \quad \quad  \quad \quad \quad \quad \quad \quad \quad \quad \quad-\int_\epsilon^t\int_{\T^N} |u(\tau)|^pdxd\tau \bigg] dx \\
&& \ \quad  +\int_{\T^N} \varphi(t) \left(|u(t)|^p-v(t) \cdot \nabla u(t) \right)dx-\int_{\T^N} |u(t)|^pdx \cdot \int_{\T^N} \varphi(t)dx.
\end{eqnarray*}
Using \eqref{20210426eq01}, we see that the above term gives
\begin{eqnarray*}
&&=\int_{\T^N} \left(\partial_t \varphi(t)+\Delta \varphi(t) \right) u(t) dx+\int_{\T^N} \varphi(t) \left(|u(t)|^p-v(t) \cdot \nabla u(t) \right)dx \\
&& \ \quad -\int_{\T^N} |u(t)|^pdx \cdot \int_{\T^N} \varphi(t)dx.
\end{eqnarray*}
To this end, we integrate the above expression over $(\epsilon, t)$ for any $\epsilon<t<T$:
\begin{eqnarray} \label{20210427eq01}
&&\int_{\T^N} \varphi(t)u(t)dx-\int_{\T^N} \varphi(\epsilon) u(\epsilon)dx=\int_\epsilon^t \int_{\T^N} \partial_t \varphi(\tau)u(\tau)dxd\tau \nonumber \\
&& \quad \quad -\int_\epsilon^t \int_{\T^N} \nabla \varphi(\tau) \cdot \nabla u(\tau) dxd\tau+\int_\epsilon^t \int_{\T^N} \varphi(\tau) \left(|u(\tau)|^p-v(\tau) \cdot \nabla u(\tau) \right)dxd\tau \nonumber \\
&& \quad \quad -\int_\epsilon^t \left(\int_{\T^N} |u(\tau)|^p dx \right) \cdot \left( \int_{\T^N} \varphi(\tau) dx \right)d \tau
\end{eqnarray}

\medskip

\textit{Step IV: Regularity of $\partial_t u$.}

\medskip

We claim that $\partial_t u \in L^2 \left(\left[\epsilon, t \right], H^{-1} \right)$. Note that it suffices to show that all the terms in \eqref{20210427eq01} are well defined for $\varphi \in L^2\left(\left[\epsilon, t \right], H^1 \right)$ with $\partial_t \varphi \in L^2\left(\left[\epsilon, t \right], H^{-1} \right)$. The key observation here is to show that the term is well-defined
\begin{eqnarray} \label{20210428eq01}
&&\int_\epsilon^t \int_{\T^N} \varphi(\tau) \left(|u(\tau)|^p-v(\tau) \cdot \nabla u(\tau) \right)dxd\tau \nonumber \\
&& \quad \quad \quad =\int_\epsilon^t \int_{\T^N} \varphi(\tau)|u(\tau)|^pdxd\tau-\int_\epsilon^t \int_{\T^N} \varphi(\tau) v(\tau) \cdot \nabla u(\tau) dxd\tau
\end{eqnarray}
is well-defined under such an assumption. To control the second term, we simply use H\"older: 
$$
\left|\int_\epsilon^t \int_{\T^N} \varphi(\tau) v(\tau) \cdot \nabla u(\tau) dx d\tau \right| \le  \|v\|_{L^{\infty}\left(\R_+; L^{\infty} \right)} \|\varphi\|_{L^2\left(\left[\epsilon, t\right] \times \T^N  \right)} \|\nabla u\|_{L^2\left(\left[\epsilon, t\right] \times \T^N  \right)}.
$$
Next for the first term in \eqref{20210428eq01}, by Cauchy's inequality and Gagliardo-Nirenberg's inequality, we have
\begin{eqnarray} \label{20210428eq10} 
&&\left|\int_\epsilon^t \int_{\T^N} \varphi(\tau) |u(\tau)|^p dxd\tau \right| \le  \|\varphi\|_{L^2\left(\left[\epsilon, t \right] \times \T^N \right)} \cdot \left( \int_\epsilon^t \|u(\tau)\|_{L^{2p}(\T^N)}^{2p} d\tau \right)^{\frac{1}{2}} \nonumber \\
&& \quad \quad \quad \lesssim  \|\varphi\|_{L^2\left(\left[\epsilon, t \right] \times \T^N \right)} \cdot \left[\int_\epsilon^t \|\nabla u(\tau)\|_{L^2(\T^N)}^{N(p-1)} \cdot \|u(\tau)\|_{L^2(\T^N)}^{2p-N(p-1)} d\tau \right]^{\frac{1}{2}} \nonumber \\
&& \quad \quad \quad \lesssim \|\varphi\|_{L^2\left(\left[\epsilon, t \right] \times \T^N \right)} \cdot \|u\|_{X_T}^{p-\frac{N(p-1)}{2}} \cdot  \left[\int_\epsilon^t \|\nabla u(\tau)\|_{L^2(\T^N)}^{N(p-1)} d\tau \right]^{\frac{1}{2}} \nonumber \\
&& \quad \quad \quad \lesssim (t-\epsilon)^{\frac{2-N(p-1)}{4}} \|\varphi\|_{L^2\left(\left[\epsilon, t \right] \times \T^N \right)} \cdot \|u\|_{X_T}^{p-\frac{N(p-1)}{2}} \cdot \|u\|_{L^2\left(\left[\epsilon, t\right]; H^1\right)}^{\frac{N(p-1)}{2}}, \end{eqnarray}
where in the above estimate, we have used the fact that $u \in X_T$, $u \in L^2\left(\left[\epsilon, t \right], H^1 \right)$ (from Step II) and 
$N(p-1)<2$.

\medskip

\textit{Step V: Approximation and energy identity over $[\epsilon, T]$.}

\medskip

Since $u \in L^2 \left(\left[ \epsilon, t \right]; H^1 \right)$ with $\partial_t u \in L^2\left(\left[\epsilon, t\right]; H^{-1} \right)$, there exists a sequence $\{\varphi_n\} \subset C^1 \left(\left[ \epsilon, T \right]; H^1 \right)$, such that $\varphi_n \to u$ strongly in $C\left(\left[ \epsilon, T \right]; L^2 \right) \cap L^2 \left( \left[\epsilon, T \right]; H^1 \right)$ as $n \to \infty$ and such that $\partial_t \varphi_n \to \partial_t  u$ weakly in $L^2 \left( \left[\epsilon, T \right]; H^{-1} \right)$. Note that by H\"older's inequality and the proof of \eqref{20210428eq10}, we have
\begin{eqnarray*}
&& \left| \int_{\epsilon}^T \int_{\T^N} \left(u(\tau)-\varphi_n(\tau)\right) |u(\tau)|^p dxd\tau \right| \\
&& \quad \quad \quad \le \left\| u-\varphi_n \right\|_{L^2 \left(\left[\epsilon, T \right] \times \T^N \right)} \cdot \left( \int_\epsilon^t \|u(\tau)\|_{L^{2p}(\T^N)}^{2p} d\tau \right)^{\frac{1}{2}} \\
&& \quad \quad \quad \lesssim  T^{\frac{2-N(p-1)}{4}} \left\| u-\varphi_n \right\|_{L^2 \left(\left[\epsilon, T \right] \times \T^N \right)}  \cdot \|u\|_{X_T}^{p-\frac{N(p-1)}{2}} \cdot \|u\|_{L^2\left(\left[\epsilon, t\right]; H^1\right)}^{\frac{N(p-1)}{2}},
\end{eqnarray*}
which clearly converges to zero as $n \to \infty$. While for the term 
$$
\left| \int_\epsilon^t \int_{\T^N} \left(\varphi_n-u(\tau) \right) v(\tau) \cdot \nabla u(\tau) dx d\tau \right|, 
$$
one can use H\"older's inequality again to see it can be controlled by 
$$
\|v\|_{L^{\infty}\left(\R_+; L^{\infty} \right)} \|\varphi_n-u\|_{L^2\left(\left[\epsilon, t\right] \times \T^N  \right)} \|\nabla u\|_{L^2\left(\left[\epsilon, t\right] \times \T^N  \right)}, 
$$
which again converges to zero as $n \to \infty$. 

\medskip

To this end, we take $\varphi_n$ as test functions in \eqref{20210427eq01}, and then pass to the limit $n \to \infty$. This gives
\begin{eqnarray} \label{20210428eq05}
&&\|u(t)\|_{L^2}^2-\|u(\epsilon)\|_{L^2}^2=\int_\epsilon^t \int_{\T^N} \partial_t u(\tau) u(\tau) dxd\tau \nonumber \\
&& -\int_\epsilon^t \|\nabla u(\tau)\|_{L^2}^2 d\tau+\int_\epsilon^t \int_{\T^N} u(\tau) \left(|u(\tau)|^p-v(\tau) \cdot \nabla u(\tau) \right)dxd\tau, 
\end{eqnarray}
where we have used the fact that $\int_{\T^N} u(\tau) dx=0$ for any $\tau \in [\epsilon, t]$. By the fundamental theorem of Calculus, 
$$
\int_\epsilon^t \int_{\T^N} \partial_t u(\tau) u(\tau) dxd\tau=\frac{\|u(t)\|_{L^2}^2-\|u(\epsilon)\|_{L^2}^2}{2}, \quad \epsilon<t<T.  
$$
Therefore, \eqref{20210428eq05} gives for any $0<\epsilon \le t \le T$, 
\begin{equation} \label{20210429eq01}
\|u(t)\|_{L^2}^2-\|u(\epsilon)\|_{L^2}^2=-2\int_\epsilon^t \|\nabla u(\tau)\|_{L^2}^2d \tau+2\int_\epsilon^t \int_{\T^N} u(\tau) |u(\tau)|^p dxd\tau. 
\end{equation}

\medskip

\textit{Step VI: Passing $\epsilon \to 0$.}

\medskip

From \eqref{20210429eq01}, we see that
\begin{equation} \label{20210429eq02}
\int_\epsilon^t \|\nabla u(\tau)\|_{L^2}^2 d\tau \lesssim  \|u(\epsilon)\|_{L^2}^2-\|u(t)\|_{L^2}^2+\int_\epsilon^t \int_{\T^N} |u(\tau)|^{p+1} dxd\tau.
\end{equation} 
Note that by Gagliardo-Nirenberg's inequality, we have for any $0<\epsilon \le t \le T$, 
\begin{eqnarray*}
&&\int_\epsilon^t \int_{\T^N} |u(\tau)|^{p+1} dxdt\tau=\int_\epsilon^t \left\|u(\tau)\right\|_{p+1}^{p+1} d\tau \\
&& \quad \quad \quad \quad \le   \int_\epsilon^t \left\|\nabla u(\tau) \right\|_{L^2}^{\frac{N(p-1)}{2}} \left\|u\right\|_{L^2}^{p+1-\frac{N(p-1)}{2}} d\tau \\
&& \quad \quad \quad \quad \le  \|u\|_{X_T}^{p+1-\frac{N(p-1)}{2}}  \cdot \int_\epsilon^t \tau^{-\frac{N(p-1)}{4}} \cdot  \left(\tau^{\frac{1}{2}} \left\|\nabla u(\tau) \right\|_{L^2}\right)^{\frac{N(p-1)}{2}} d\tau  \\
&& \quad \quad \quad \quad \le C_{T, p, N} \|u\|_{X_T}^{p+1},
\end{eqnarray*}
for some absolute constant $C_{T, p, N}$ only depending on $T, p$ and $N$. Note that in the above estimate, we have used the fact that $0<\frac{N(p-1)}{4}<1$. 

Therefore, by Lebesgue's Dominated convergence theorem, we are able to let $\epsilon \to 0$ in \eqref{20210429eq01}, which proves the desired energy identity \eqref{energyid}. Here, in particular we are using the fact that $\|u(\epsilon)\|_{L^2} \to \|u(0)\|_{L^2}$, which can be easily verified via the definition of the mild solution. Similarly, since $u \in L^2\left((0, T), H^1 \right)$ (by \eqref{20210429eq02} and monotone convergence theorem), the right hand side of \eqref{20210427eq01} with $\varphi \in C_c^\infty\left(\left[0, T\right) \times \T^N \right)$ is uniformly bounded in $\epsilon$, so the left hand side is, which implies $(\partial_t u) \one_{[\epsilon, T)}$ converges weakly to $\partial_t u$ in $L^2((0, T); H^{-1})$. Therefore, by the Lebesgue's Dominated convergence theorem again, we can pass to the limit $\epsilon \to 0$ in \eqref{20210427eq01}, which gives the weak formulation \eqref{weaksol}. Hence $u$ is a weak solution on $[0, T]$. 
\end{proof}

\bigskip
\section{Global existence with advecting drifts of small dissipation time} \label{Sec3}
As an application of our local existence and regularity results, in this section, we show that if the advecting velocity $v$ has small \emph{dissipation time} (see, Definition \ref{e:ADE}),  then the solution to~\eqref{maineq} converges to a homogeneous mixing state exponentially fast, which further guarantees the global existence of the mild (also weak) solution to the problem ~\eqref{maineq}. It turns out such an existence result can be handled by the general theory developed in \cite[Theorem 1.2]{IXZ21}, and we include the details here for the purpose of being self-contained. 

To begin with, we put stronger assumptions on the drift term: we assume the incompressible velocity $v$ is Lipschitz continuous in space uniformly in time, i.e. $v\in L^{\infty}\left(\left[0,\infty\right),W^{1,\infty}\left(\T^{N}\right)\right)$. For simplicity, we also assume $\bar{u}=0$, while for the case when $\bar{u} \neq 0$, one can modify the proof by considering $\widetilde{u}=u-\bar{u}$. 

Let $\calS_{s,t}, 0\le s\le t$, be the solution operator of the advection-diffusion equation
\begin{equation}\label{e:ADE}
\partial_t\theta + v\cdot\nabla\theta-\Delta\theta=0\, ,
\end{equation}
that is $\theta\left(t\right)=\calS_{s,t}\theta(s)$.

Next, we introduce the dissipation time of~\eqref{e:ADE}. Heuristically, it is the smallest time that the system halve the energy, $L^2$ norm, of the solution. The formal definition is given below:
\begin{defn}
\label{def:dissipationTime}
The \emph{dissipation time} associated to the solution operator $\calS_{s, t}$, $0\le s\le t$ is defined by
  \begin{equation}
  \label{e:dtimeDef}
    \tau_*(v):=\inf \left\{ t \geq 0 \Big\vert  \| \calS_{s, s+t} \|_{ L_0^2 \to  L_0^2 } \le \tfrac{1}{2} \text{ for all $s\geq 0$} \right\}\,.
  \end{equation}

\end{defn}
We now describe the work in \cite{IXZ21}. Consider the following nonlinear advection-diffusion equation: 
\begin{equation}\label{e:ADNE}
\partial_t\theta + v\cdot\nabla\theta-\Delta\theta=\calN\left(\theta\right)
\end{equation}
on $\T^{N}$ and with initial data $\theta_0\in L_0^2(\T^N)$, where $L_0^2(\T^N)$ denotes the collection of all $L^2$ integrable functions on $\T^2$ with mean $0$. The velocity $v$ in the advection term is a Lipschitz divergence-free vector field with bounded Lipschitz constant. The non-linear operator $\calN:H^{1}\left(\T^N\right)\rightarrow L_0^2\left(\T^N\right)$ is measurable, and its target space being $L_0^2$ implies the solution to~\eqref{e:ADNE} remain mean-zero. The authors proved following main theorem:
\begin{thm}\label{thm:IXZthm}
Assume that $\calN$ satisfies the following hypotheses $(H1)-(H2)$:
\begin{enumerate}
    \item [(H1)] There exists $\epsilon_0\in\left(0,1\right]$ and an increasing continuous function $F:\left[0,\infty\right)\rightarrow\left[0,\infty\right)$ such that for every $\varphi\in H^1(\T^N)$ we have
    \begin{equation*}
        \Big\vert\int_{\T^d}\,\varphi\calN\left(\varphi\right)\,dx\Big\vert
        \le (1-\epsilon_0)\Vert\nabla\varphi\Vert_{L^2}^2 + F(\Vert\varphi\Vert_{L^2})\, .
    \end{equation*}
    
    \medskip
    
    \item [(H2)] There exists $C_0<\infty$ and an increasing continuous function $G:\left[0,\infty\right)\rightarrow\left[0,\infty\right)$ such that for every $\varphi\in H^1(\T^N)$ we have
    \begin{equation*}
        \Vert\calN(\varphi)\Vert_{L^2}\le C_0\Vert\nabla\varphi\Vert_{L^2}^2 + G(\Vert\varphi\Vert_{L^2})\, .
    \end{equation*}
\end{enumerate}
and 
\begin{equation}
    \theta\in L_{\text{loc}}^2\left(\left(0,T\right),H^1\left(\T^N\right)\right)\cap C\left(\left[0,T\right),L_0^2\left(\T^N\right)\right)
\end{equation}
is a weak solution to~\eqref{e:ADNE} with $v\in L^{\infty}\left(\left[0,\infty\right),W^{1,\infty}\left(\T^{N}\right)\right)$ a divergence free vector field. Then there is $\tau_0=\tau_0\left(\Vert\theta_0\Vert_{L^2},\calN\right)$ such that if $\tau_{*}(v)\le \tau_0$, one has
\begin{equation*}
    \sup_{t\in[0,T)}\Vert\theta(t)\Vert_{L^2}
    \le 2\Vert\theta_0\Vert_{L^2}+1\,.
\end{equation*}
Moreover, if the functions $F$ and $G$ from hypotheses $(H1)$ and $(H2)$ respectively, satisfy
\begin{equation}\label{e:FGcond}
    \limsup_{y\rightarrow0^{+}}\frac{\sqrt{F(y)}+G(y)}{y}<\infty
\end{equation}
and $T=\infty$, then $\Vert\theta(t)\Vert_{L^2}\rightarrow 0$ exponentially as $t\rightarrow\infty$.
\end{thm}
\begin{rem}
The threshold value $\tau_0$ and the exponential decay rate can be computed explicitly in terms of $F,G,C_0,\epsilon_0$ and $\Vert\theta_0\Vert_{L^2}$ through the proof of the theorem.
\end{rem}
Notice that equation~\eqref{maineq} can be formulated into~\eqref{e:ADNE} with $\calN(u)=|u|^p-\int_{\T^N} |u|^p$. Therefore to prove the global existence of~\eqref{maineq} for $v$ with small dissipation time, it suffices for us to check $\calN(u)$ satisfies the hypotheses in Theorem~\ref{thm:IXZthm}, and the formal conclusion is summarized in the following theorem.
\begin{thm}\label{thm:expdecay of maineqn}
Let $\gamma>0$ and $u_0 \in L_0^2 (\T^N)$. Then for $1<p<1+\frac{4}{N}$ there exists a threshold $\tau_0=\tau_0(\Vert u_0\Vert_{L^2},p,\gamma)$, such that if $\tau_*(v)\le\tau_0$, then there exists a constant $C_0=C_0(\|u_0\|_{L^2}, p, \gamma)>0$, such that 
\begin{equation}
    \Vert u(t)\Vert_{L^2}\le C_0e^{-\gamma t}\Vert u_0\Vert_{L^2}\,.
\end{equation}
on $[0,T]$.
\end{thm}

As a consequence of Corollary \ref{cor01} and Theorem~\ref{thm:expdecay of maineqn}, we can easily obtain the global existence of mild solutions to~\eqref{maineq}:
\begin{cor} \label{20210814cor01}
With the same assumptions of Theorem~\ref{thm:expdecay of maineqn}, the mild solution $u$ of~\eqref{maineq} exists on $[0,\infty)$.
\end{cor}
\begin{rem}\label{examples}
Here we would like to include several examples when the dissipation time $\tau_*(v)$ can be arbitrarily small. 
\begin{enumerate}
\item [(1).] Our first example is the \emph{mixing flow}. Recall that from classical theories of dynamical system (see, for instance [Ergodic theory, 5 in annals]), if $v$ is said to be \emph{strongly mixing} with rate function $h$ if for every $s, t \ge 0$ and $\phi_0, \psi \in H^1(\T^N)$ with $\int_{\T^N} \phi_0 dx=\int_{\T^N} \psi dx=0$, one has
$$
\left| \langle \phi(s+t), \psi \rangle \right| \le h(t) \|\phi(s)\|_{H^1} \|\psi\|_{H^1}. 
$$
Here, $h(t): [0, \infty) \to (0, \infty)$ is a continuous decreasing function that vanishes at $\infty$ and $\phi(s):=e^{s v \cdot \nabla} \phi_0$. \cite[Proposition 1.4]{FFIT19} suggests the following: let $v_A(x, t):=Av(x, At)$, then
$$
\tau_*(v_A) \to 0 \quad \textrm{as} \quad A \to \infty; 
$$

\medskip

\item [(2).] Our second example is the \emph{cellular flow}, which is given by
\begin{equation}\label{e:sinflow}
  v_l^A(x):
    =A \nabla^\perp \left(\frac{1}{l}\sin(\frac{2\pi}{l} x)\sin(\frac{2\pi}{l} y) \right)
    = 2\pi A \begin{pmatrix} -\sin( \frac{2\pi}{l} x) \cos( \frac{2\pi}{l} y) \\
    \\
      \cos(\frac{2\pi}{l} x) \sin( \frac{2\pi}{l} y )
    \end{pmatrix}\ , 
\end{equation}
where $\ell, A>0$. Indeed, one can check that $\tau_*\left(v_l^A\right)$ can be arbitrarily small with $l$ sufficiently small and $A$ sufficiently large and we refer the reader \cite{IXZ21} for more detailed information. 

\medskip

To our best knowledge, Corollary \ref{20210814cor01} hence provides the first example rather than Keller-Segel equations which satisfy the assumptions in \cite{IXZ21} such that one can use cellular flow to prevent the blowup of a physical equation.
\end{enumerate}
\end{rem}
\begin{rem}
Corollary \ref{20210814cor01} and Remark \ref{examples} together conclude Theorem \ref{thm11}. 
\end{rem}

Now we prove Theorem~\ref{thm:expdecay of maineqn} by checking the nonlinear term $\calN(u)=|u|^p-\int_{\T^N} |u|^p dx$ satisfies the hypotheses in Theorem~\ref{thm:IXZthm}.
\begin{proof} [Proof of Theorem \ref{thm:expdecay of maineqn}]
\textit{Hypotheses $(H1)$:} For $1<p<1+\frac{4}{N}$ we have:
\begin{eqnarray*}
&&\Big\vert\int_{\T^N}\,u\calN\left(u\right)\,dx\Big\vert=\Big\vert\int_{\T^N}\,u\vert u\vert^p\,dx\Big\vert \\
&& \quad \quad \quad \quad \le \Vert u\Vert_{L^{p+1}}^{p+1} \\
&& \quad \quad \quad \quad \lesssim \Vert u\Vert_{L^2}^{(p+1)-(\tfrac{p-1}{2})N} \Vert \nabla u\Vert_{L^2}^{(\tfrac{p-1}{2})N}  \\
&& \quad \quad \quad \quad \lesssim \frac{1}{2}\Vert \nabla u\Vert_{L^2}^2 + C\Vert u\Vert_{L^2}^{\tfrac{4(p+1)-2(p-1)N}{4-(p-1)N}},
\end{eqnarray*}
Thus $(H1)$ is satisfied with $\epsilon_0=\frac{1}{2}$ and $F(y)=Cy^{\tfrac{4(p+1)-2(p-1)N}{4-(p-1)N}}$.

\medskip

\textit{Hypotheses $(H2)$:}
\begin{equation*}
\Big\Vert\calN\left(u\right)\Big\Vert_{L^2}\le\Vert u\Vert_{L^{2p}}^{p}+\Vert u\Vert_{L^p}^{p}\,,
\end{equation*}
For the first term, by the Gagliardo-Nirenberg and Young's inequalities, for $1<p<1+\frac{4}{N}$ we have:
\begin{equation*}
    \Vert u\Vert_{L^{2p}}^p
    \lesssim \Vert\nabla u\Vert_{L^2}^{\tfrac{(p-1)N}{2}} \Vert  u\Vert_{L^2}^{p-\tfrac{(p-1)N}{2}}
    \lesssim \frac{C_0}{2}\Vert \nabla u\Vert_{L^2}^2+C_1\Vert u\Vert_{L^2}^{\tfrac{4p-2(p-1)N}{4-(p-1)N}}\,.
\end{equation*}
Similarly, for the second term we have:
\begin{equation*}
    \Vert u\Vert_{L^{p}}^p
    \lesssim \Vert\nabla u\Vert_{L^2}^{(\frac{p}{2}-1)N} \Vert  u\Vert_{L^2}^{p-(\frac{p}{2}-1)N}
    \lesssim \frac{C_0}{2}\Vert \nabla u\Vert_{L^2}^2+C_2\Vert u\Vert_{L^2}^{\tfrac{4p-2(p-2)N}{4-(p-2)N}}\,.
\end{equation*}
Thus $(H2)$ is satisfied with $G(y)=C_1y^{\tfrac{4p-2(p-1)N}{4-(p-1)N}}+C_2y^{\tfrac{4p-2(p-2)N}{4-(p-2)N}}$. Moreover, one can further verify that $F,G$ satisfy~\eqref{e:FGcond}. The proof is complete. 
\end{proof}

\bigskip

\section{Applications to shear flow} \label{Sec4} 

In this section, we consider an application of Corollary \ref{cor01} to the case when $v$ becomes a horizontal shear flow. More precisely, we consider the case when $N=2$ and the flow ${\bf v}$ is given by ${\bf v}=(v_1(x_2), 0)$:
$$
\partial_t u+Av_1(x_2) \partial_{x_1} u-\Delta u=|u|^p-\int_{\T^2} |u|^p,
$$
where the parameter $A>0$ is the amplitude of the flow, $1<p<2$, $v_1 \in L^\infty(\T)$ and $x=(x_1, x_2) \in \T^2$. By a change of time, we can write the above equation by 
\begin{equation} \label{sheareq}
\partial_t u+v_1(x_2) \partial_{x_1} u-\nu \Delta u=\nu |u|^p-\nu \int_{\T^2}|u|^p,
\end{equation} 
where $\nu=A^{-1}$. Here, $\nu$ is the \emph{viscosity coefficient}, which measures the strength of the dissipation. The goal is to show the global existence of the equation \eqref{sheareq} under proper assumption on ${\bf v}$ and initial condition $u_0$.

For any $g \in L^2(\T^2)$, we denote 
$$
\langle g \rangle(t, x_2)=\int_{\T} g(t, x_1, x_2)dx_1 \quad \textrm{and} \quad g_{\notparallel}(t, x_1, x_2)=g(t, x_1, x_2)-\langle g \rangle(t, x_2).
$$
Note that
$$
\langle g \rangle \in \textrm{Ker} \left(v_1(x_2)\partial_{x_1} \right) \quad \textrm{and} \quad g_{\notparallel} \in \big(\textrm{Ker} \left(v_1(x_2)\partial_{x_1} \right) \big)^{\perp}. 
$$
We have the following property on shear flow ${\bf v}$.

\begin{prop}{\cite{CDE20}}\label{prop41}
If $v_1$ has only finitely many critical points with order at most $m$, namely, at most $m-1$ derivatives vanishes at these critical points. Then there exists some constant $C>0$, such that 
\begin{equation} \label{20210518eq14b}
\|e^{-\left(v_1\partial_{x_1} \right)t} g_{\notparallel} \|_{H^{-1}} \le \frac{C}{(1+t)^m} \|g_{\notparallel}\|_{H^1}, \quad t \ge 0, 
\end{equation} 
for any $g \in L^2_0(\T^2)$, or $L^2$ function with zero mean.
\end{prop}

A direct consequence of the mixing condition above is an enhanced dissipation estimate for its corresponding advection-diffusion model:

\begin{prop}{\cite[Theorem 2.1]{CDE20}}\label{prop42}
For each $\nu>0$ and ${\bf v}=(v_1(x_2), 0)$, $v_1$ satisfies the assumption on Proposition \ref{prop41}, define
\begin{equation} \label{20210512eq111}
H_{\nu}:=-\nu \Delta+v_1(x_2) \partial_{x_1}.
\end{equation} 
Then for any $g \in L^2_0(\T^2)$, one has for any $t \ge 0$,
\begin{equation} \label{20210512eq01}
\|e^{-t H_\nu} g_{\notparallel}\|_{L^2} \le 10e^{-\lambda_\nu t} \|g_{\notparallel}\|_{L^2}, \quad \lambda_\nu:=c_0 \nu ^{\frac{2}{2+m}},
\end{equation}
where $c_0>0$ is some absolute constant independent of the choice of $\nu$. 
\end{prop}

\begin{thm} \label{mainthm03}
Let $u_0 \in L^2_0(\T^2)$ satisfy the following conditions:\\
The $L_{x_2}^2$-norm of the average of $u_0$ along $x_2$-direction is sufficiently small, that is, 
\begin{equation} \label{20210518eq14}
0 \le \left\| \langle u_0 \rangle \right\|_{L_{x_2}^2} \ll 1.
\end{equation} 

Let further, ${\bf v} \in L^\infty(\T^2): [0, 1)^2 \to  \R$ be a shear flow satisfies the assumption on Proposition \ref{prop41} with $m \ge 2$. Then there exists $0<\nu_0<1$ depending on $v_1$ and $\|u_0\|_{L^2}$ with the following property: for any $0<\nu<\nu_0$, there exists a global-in-time weak solution $u$ of \eqref{sheareq} with initial data $u_0$ such that $u \in L^\infty([0, \infty); L^2(\T^2)) \cap L^2([0, \infty); H^1(\T^2))$. 
\end{thm}

\begin{rem} \label{20210517rem01}
\begin{enumerate}
\item [(1).] In the later proof, we will have a quantitative bound on describing how small the integral in \eqref{20210518eq14} is (see, \eqref{20210518eq15}); 

\medskip

\item [(2).] Note that we only require along the $x_1$-direction, the flow ${\bf v}$ is mixing, while along $x_2$-direction the flow ${\bf v}$ is identically equal to zero and hence ${\bf v}$ itself is not a mixing flow. Therefore, Theorem \ref{mainthm03} is independent of the results in Section 3 and of its own interest; 

\medskip

\item [(3).] The condition \eqref{20210518eq14} is natural and dates back to a comprehensive study of the non-local semi-linear parabolic equation \eqref{withoutv} with Neumann boundary condition by El Soufi, Jazar and Monneau \cite[Theorem 1.5]{SJM07}. Nevertheless, we make a remark that our assumption \eqref{20210518eq14} here is slightly different from the one in \cite{SJM07}, which requires the $L^\infty$-norm of the initial data being sufficiently small. 

\medskip

\item [(4).] We observe that Theorem \ref{mainthm03} holds locally from the arguments on Section \ref{sec2}. Indeed, by Theorem \ref{mainthm04}, we have seen that at least for a sufficiently small time $t_0>0$, the weak solution $u$ of \eqref{sheareq} exists and
\begin{equation} \label{20210511eq01}
u \in L^\infty([0, t_0]; L^2(\T^2)) \cap L^2([0, t_0]; H^1(\T^2))
\end{equation} 
with $\partial_t u \in L^2([0, t_0]; H^{-1}(\T^2))$. Moreover, Corollary \ref{cor01} suggests that the mild and weak solution persists as long as its $L^2$ norm is finite. In this section, to prove Theorem \ref{mainthm03}, we only need to show $\|u\|_{L^2}$ is bounded uniformly in time. 
\end{enumerate}
\end{rem}

\medskip

\subsection{Decomposition of the solution}

We let $u$ be the weak solution as in Theorem \ref{mainthm04}, and recall that 
$$
\langle u \rangle(t, x_2)=\int_{\T} u(t, x_1, x_2)dx_1 \quad \textrm{and} \quad u_{\notparallel}(t, x_1, x_2)=u(t, x_1, x_2)-\langle u \rangle(t, x_2). 
$$
By \eqref{20210511eq01}, it is easy to see that
\begin{equation} \label{20210511eq02}
\langle u \rangle \in L^\infty([0, t_0]; L^2(\T)) \cap L^2([0, t_0]; H^1(\T))
\end{equation}
and 
\begin{equation} \label{20210511eq03}
u_{\notparallel} \in L^\infty([0, t_0]; L^2(\T^2)) \cap L^2([0, t_0]; H^1(\T^2)). 
\end{equation} 

Moreover, formally, by integrating \eqref{sheareq} in $x_1$ and subtracting it from \eqref{sheareq}, one may notice that $\langle u \rangle$ and $u_{\notparallel}$ satisfy the following system of coupled equations: 
\begin{equation} \label{eqpar}
\partial_t \langle u \rangle -\nu \partial_{x_2}^2 \langle u \rangle=\nu \int_{\T} \left|\langle u \rangle+u_{\notparallel}\right|^pdx_1-\nu \int_{\T^2} \left|\langle u \rangle+u_{\notparallel}\right|^pdx_1dx_2
\end{equation}
and
\begin{equation} \label{eqnotpar}
\partial_t u_{\notparallel}+v_1(x_2)\partial_{x_1} u_{\notparallel}-\nu \Delta u_{\notparallel}=\nu \left|\langle u \rangle+u_{\notparallel} \right|^p-\nu \int_{\T} |\langle u \rangle+u_{\notparallel}|^p dx_1.  
\end{equation} 
Nevertheless, parallel to Definition \ref{mildsol}, we also have the two corresponding parts of mild solution to \eqref{sheareq} satisfies the following equations:
\begin{align*} 
\label{mildsoladpara}
\langle u \rangle (t)&= e^{-t\partial^2_{x_2}} \left(\langle u \rangle(0) \right)+\nu \int_0^t e^{-(t-\tau)\partial^2_{x_2}} \left( \int_{\mathbb{T}}|\langle u \rangle(\tau)+u_{\notparallel}(\tau)|^pdx_1\right) \nonumber \\ &\quad-\nu\int_0^t\int_{\T^2}  |\langle u \rangle(\tau)+u_{\notparallel}(\tau)|^p  dx_1dx_2  d\tau 
\end{align*} 
and
\begin{equation} \label{20210512eq201}
u_{\notparallel}(t)= \calS_{t-s} \left(u_{\notparallel}(s) \right)+\nu \int_s^t \calS_{t-\tau} \left( |\langle u \rangle(\tau)+u_{\notparallel}(\tau)|^p   -\int_\T  |\langle u \rangle(\tau)+u_{\notparallel}(\tau)|^p  dx_1 \right) d\tau, 
\end{equation} 
where $\calS_t$ is the solution operator from $0$ to time $t \ge 0$ for the advection-diffusion equation
$$
\partial_t g+v_1(x_2) \partial_{x_1} g- \nu \Delta g=0,
$$
namely, $\calS_t=e^{-tH_\nu}$, where $H_{\nu}$ is defined in \eqref{20210512eq111}. 

%


\subsection{Bootstrap assumption} The goal of the second part of this section is to prove some prior estimates for $u_{\notparallel}$, which can be thought as the starting point for our bootstrap argument. 

Recall that $t_0>0$ is defined as in \eqref{20210511eq01}. We set up our prior bootstrap assumptions via the following facts. 

\begin{lem}
Under the assumption of Proposition \ref{prop42}, there exists a sufficiently small time $0<\widetilde{t}_{0, 1} \le t_0$, such that for any $0 \le s<t \le \widetilde{t}_{0, 1}$, there holds
\begin{equation} \label{20210521eq301}
\|u_{\notparallel} (t)\|_{L^2} \lesssim e^{-\frac{\lambda_\nu (t-s)}{4}} \left\|u_{\notparallel}(s) \right\|_{L^2}. 
\end{equation} 
Here, $\widetilde{t}_{0, 1}$ can depend on $\nu$, $p$, $\|u\|_{X_{t_0}}$, $\|u\|_{L^2([0, t_0]; H^1)}$, ${\bf v}$ and $\|u_0\|_{L^2}$, where we recall that $u=\langle u \rangle+u_{\notparallel}$ and $u_0=u(0)$, and the implicit constant in  \eqref{20210521eq301} only depends on the dimension. 
\end{lem}

\begin{proof}
Taking $L^2$ norm on both sides of \eqref{20210512eq201}, we have
\begin{eqnarray*}
\|u_{\notparallel}(t)\|_{L^2}%
&\lesssim& \|\calS_{t-s}(u_{\notparallel}(s))\|_{L^2}+\nu (t-s)^{\frac{1}{2}} \left(\int_s^t \|u(\tau)\|_{L^{2p}}^{2p} d\tau\right)^{\frac{1}{2}} \\
&\lesssim& e^{-\lambda_\nu(t-s)}\|u_{\notparallel}(s)\|_{L^2}+ \nu (t-s)^{\frac{3-p}{2}} \cdot \|u\|_{X_{t_0}} \|u\|_{L^2([0, t_0], H^1)}^{p-1}, 
\end{eqnarray*}
where in the last inequality, we have used the estimate \eqref{20210512eq01} and the proof of \eqref{20210428eq10} with $N=2$, respectively. The desired claim then follows by letting $\widetilde{t}_{0, 1}$ be sufficiently small. 
\end{proof}

\begin{lem}
There exists a sufficiently small time $0<\widetilde{t}_{0, 2} \le t_0$, such that for any $0 \le s<t \le \widetilde{t}_{0, 2}$, there holds
\begin{equation} \label{20210531eq02}
\nu \int_s^t \|\nabla u_{\notparallel}(\tau) \|^2_{L^2} d\tau \lesssim \|u_{\notparallel}(s)\|_{L^2}^2.
\end{equation} 
Here, $\widetilde{t}_{0, 2}$ can depend on $\nu$, $p$, $\|u\|_{X_{t_0}}$, $\|u\|_{L^2([0, t_0]; H^1)}$, and the implicit constant in \eqref{20210531eq02} only depends on the dimension. 
\end{lem}

\begin{proof}
By the energy identity corresponding to $u_{\notparallel}$, we see that for any $0 \le \tau \le t_0$, 
\begin{eqnarray} \label{20210519eq11}
&&\frac{1}{2} \frac{d}{dt} \|u_{\notparallel}(\tau)\|_{L^2}^2+\nu \|\nabla u_{\notparallel}(\tau)\|_{L^2}^2 \nonumber \\
&& \quad \quad \quad \quad = \nu \int_{\T^2} u_{\notparallel}(\tau) \left|\langle u \rangle(\tau)+u_{\notparallel}(\tau)\right|^pdx_1dx_2 \nonumber \\
&& \quad \quad \quad \quad \quad \quad \quad \quad \quad \quad - \nu \int_{\T} \left[\int_{\T} \left|\langle u \rangle(\tau)+u_{\notparallel}(\tau) \right|^p dx_1 \cdot \int_{\T} u_{\notparallel}(\tau) dx_1 \right] dx_2 \nonumber \\
&& \quad \quad \quad \quad \le 2\nu \|u_{\notparallel}(\tau)\|_{L^2} \cdot \|u(\tau)\|_{L^{2p}}^{p} \le \nu \left\|u_{\notparallel}(\tau)\right\|_{L^2}^2+\nu \|u(\tau)\|_{L^{2p}}^{2p}. 
\end{eqnarray}
Integrating the above estimate with respect to the interval $[s, t]$, we have
\begin{eqnarray*}
&& \frac{\|u_{\notparallel}(t)\|_{L^2}^2}{2}+\nu \int_s^t \|\nabla u_{\notparallel} (\tau) \|_{L^2}^2 d\tau \\
&& \quad \quad \quad  \le \frac{\|u_{\notparallel}(s)\|_{L^2}^2}{2}+\nu \int_s^t \left( \left\|u_{\notparallel}(\tau)\right\|_{L^2}^2+ \|u(\tau)\|_{L^{2p}}^{2p} \right) d\tau\\
&& \quad \quad \quad  \lesssim \frac{\|u_{\notparallel}(s)\|_{L^2}^2}{2}+\nu (t-s) \|u\|_{X_{t_0}}^2+\nu (t-s)^{2-p} \cdot \|u\|_{X_{t_0}}^2 \|u\|_{L^2([0, t_0], H^1)}^{2(p-1)},
\end{eqnarray*}
where in the last estimate, we have used the proof of \eqref{20210428eq10} with $N=2$ again. Similarly, the desired estimate \eqref{20210531eq02} holds if $\widetilde{t}_{0, 2}$ is sufficiently small. 
\end{proof}

The estimates \eqref{20210521eq301} and \eqref{20210531eq02} suggests that we can make the following \emph{bootstrap assumptions}: for any $\nu>0$ and $0 \le s \le t \le \widetilde{t}_0$,  
\begin{enumerate}
    \item [(1).] $\|u_{\notparallel}(t)\|_{L^2} \le 20e^{-\frac{\lambda_\nu(t-s)}{4}} \|u_{\notparallel}(s)\|_{L^2}$; 
    
    \medskip
    
    \item [(2).] $\nu \int_s^t \|\nabla u_{\notparallel}(\tau)\|_{L^2}^2 d\tau \le 10 \|u_{\notparallel}(s)\|_{L^2}^2$. 
\end{enumerate}
Here $\widetilde{t}_0>0$ is the maximal time such that the estimates above hold on $[0, \widetilde{t}_0]$.

\begin{rem}
The exact values of the coefficients in the above assumption are not important, while what is important for us is that once such a choice is fixed,  then if we assume $\nu$ is sufficiently small, such an assumption can improve itself by replacing the coefficients $20$ and $10$, respectively, into smaller ones, which, together with Corollary \ref{cor01}, will ``force" $t_0$ to take the value $\infty$. 
\end{rem}

\medskip
\subsection{Uniform bounds of $\langle u \rangle$} In the third part of this section, we are aiming to show the terms 
\begin{equation} \label{20210518eq13}
\left\|\langle u \rangle (t) \right\|_{L_{x_2}^2}^2 \quad \textrm{and} \quad \nu \int_0^t \left\|\partial_{x_2} \langle u \rangle(\tau) \right\|_{L_{x_2}^2}^2 d\tau
\end{equation}
can be bounded above uniformly for all $t \in [0, \widetilde{t}_0]$, under the bootstrap assumption and sufficiently small $\nu$. 

 Before we proceed, we would like include a comparison between our approach and the one in \cite{CDFM21} by Colti Zeltati, Dolce, Feng and Mazzucato, in which, they considered the two-dimensional Kuramoto-Sivashinsky equation with a shear flow. Their study of the average part of the solution along the $x_2$-direction (namely, let $\psi=\partial_{x_2}\langle \phi \rangle$) reduces to study the following $1$-dimensional ODE
$$
\partial_t \psi+\frac{\nu}{2} \int_\T \partial_{x_2} \left|\nabla \phi_{\notparallel}\right|^2dx_1+\nu \psi \partial_{x_2} \psi+\nu \partial_{x_2}^4 \psi+\nu \partial_{x_2}^2 \psi=0,
$$
whose energy estimate has the form:
\begin{equation} \label{20210517eq01}
\frac{1}{2} \frac{d}{dt} \|\psi\|_{L_{x_2}^2}^2+\nu \|\partial_{x_2}^2 \psi\|^2_{L_{x_2}^2}=\nu \|\partial_{x_2} \psi\|_{L_{x_2}^2}^2+\frac{\nu}{2} \int_{\T} \left| \nabla \phi_{\notparallel} \right|^2 \partial_{x_2} \psi dx_1.
\end{equation} 
One crucial step to handle the above estimate is to use the one-dimensional Poincar\'e's inequality to absorb term $\|\partial_{x_2} \psi\|_{L_{x_2}^2}^2$ by $\|\partial_{x_2}^2 \psi\|_{L_{x_2}^2}^2$, and this leads to a uniform bound of the term $\|\psi\|_{L_{x_2}^2}^2$ by the initial data (see, \cite[Section 2]{CDFM21}). To our best knowledge, this approach cleverly uses the full strength of the structure of the Kuramoto-Sivashinsky equation, however, it might not work for other situation, for example, if the term $\partial_{x_2}^2 \psi$ is replaced by $\psi \partial_{x_2}^2 \psi$. 

A similar situation also happens for the energy estimate of equation \eqref{eqpar}, in which and heuristically, the term $\nu \|\partial_{x_2} \psi\|_{L_{x_2}^2}^2$ is replaced by $\nu \left\| \langle u \rangle \right\|_{L_{x_2}^{2p}}^{2p}$, which suggests that the Poincar\'e technique might not work directly. Here we overcome this difficulty by applying a version of \emph{small energy method} to shear flows, and our method is motivated by \cite[Section 5]{SJM07}. 

We now turn to the details and we begin with the following prior estimate. 

\begin{lem} [A prior estimate] \label{20210517prop011}
Assume the bootstrap assumption and moreover, there exists a positive number $0<T \le \widetilde{t}_0$, such that
\begin{equation} \label{20210518eq04}
\int_\T \langle u_0\rangle (x_2) dx_2=0
\end{equation} 
and 
\begin{equation} \label{20210518eq01}
\| \langle u \rangle (t,\cdot)\|_{L_{x_2}^2} \le K, \quad \textrm{for all} \quad t \in [0, T], 
\end{equation}
where 
\begin{equation} \label{20210518eq03}
K \le \left( \frac{\lambda_1}{4C_p} \right)^{\frac{5-p}{4(p-1)}}.
\end{equation} 
Here, $\lambda_1$ is the smallest positive eigvenvalue of $-\Delta$ on $\T$ and $C_p$ is an absolute constant which only depends on $p$. 

Then the following estimate holds
\begin{equation} \label{20210517eqpri}
\| \langle u \rangle (T, \cdot )\|_{L_{x_2}^2} ^2 \le \exp \left(C_{p, \nu} \|u_{\notparallel}(0)\|_{L^2}^p \right) \cdot  \left[\left\| \langle u \rangle(0, \cdot) \right\|_{L_{x_2}^2}^2+C_{p, \nu} \|u_{\notparallel}(0)\|_{L^2}^p \right],
\end{equation} 
where
$$
C_{p, \nu}:=20C_p \cdot 10^{\frac{p-1}{2}} \left[2(3-p) \right]^{\frac{3-p}{2}} \cdot \left(\frac{\nu}{\lambda_\nu} \right)^{\frac{3-p}{2}}. 
$$
\end{lem}

\begin{rem} \label{20210518rem01}
Note that $C_{p, \nu}$ is increasing with respect to $\nu$ and also
$$
\lim_{\nu \to 0} C_{p, \nu}=0.
$$
These properties are clear from the definition of $\lambda_\nu$ (see, \eqref{20210512eq01}). 
\end{rem}

\begin{proof} [Proof of Lemma \ref{20210517prop011}.]
We start with commenting that the constant $C_p$ used in the proof of this result might differ from line by line, which, however, will depend only on $p$. Next observe that
\begin{equation} \label{20210522eq01}
\int_\T \langle u \rangle (t, x_2)dx_2=0 \quad \textrm{for any} \quad t \ge 0.
\end{equation} 
Indeed, by \eqref{20210518eq04}, we know $u_0$ has mean zero on $\T^2$, which implies $u(t, \cdot)$ has mean zero on $\T^2$ for any $t \ge 0$. 

Using the energy identity of $\langle u \rangle$, we find that
\begin{eqnarray} \label{20210518eq02}
&&\frac{1}{2} \cdot \frac{d}{dt} \int_\T |\langle u \rangle|^2 dx_2+\nu \int_\T \left| \partial_{x_2} \langle u \rangle \right|^2dx_2 = \nu \int_{\T^2} \langle u \rangle (x_2) \left| \langle u \rangle+u_{\notparallel} \right|^p dx_1dx_2 \nonumber  \\
&& \quad \quad \quad \le C_p \nu \int_{\T^2} \left| \langle u \rangle(x_2) \right| \left( \left| \langle u \rangle \right(x_2)|^p+ \left|u_{\notparallel}(x_1, x_2) \right|^p \right) dx_1dx_2  \nonumber\\
&& \quad \quad \quad =C_p \nu \int_\T |\langle u \rangle|^{p+1} dx_2+C_p \nu \int_{\T^2} \left| \langle u \rangle(x_2) \right| \left| u_{\notparallel} (x_1, x_2) \right|^p dx_1dx_2  \nonumber\\
&& \quad \quad \quad \le C_p \nu \int_\T |\langle u \rangle|^{p+1} dx_2+C_p \nu \|\langle u \rangle \|_{L_{x_2}^2} \left\|u_{\notparallel} \right\|_{L^{2p}}^p \nonumber \\
&& \quad \quad \quad :=A+B. 
\end{eqnarray}

\medskip

\textit{Estimation of $A$.} By the Gagliardo–Nirenberg's inequality, we see that
\begin{eqnarray*}
A%
&=& C_p \nu \left\|\langle u \rangle \right\|_{L_{x_2}^{p+1}}^{p+1} \le C_p \nu \left( \left\| \partial_{x_2}  \langle u\rangle \right\|_{L_{x_2}^2}^{\frac{p-1}{2(p+1)}} \cdot \left\| \langle u \rangle \right\|_{L_{x_2}^2}^{\frac{p+3}{2(p+1)}}\right)^{p+1} \\ 
&=& C_p \nu \left\| \partial_{x_2} \langle u \rangle\right\|_{L_{x_2}^2}^{\frac{p-1}{2}} \left\| \langle u \rangle \right\|_{L_{x_2}^2}^{\frac{p+3}{2}} \le  \frac{\nu \left\|\partial_{x_2} \langle u \rangle \right\|_{L_{x_2}^2}^2}{2}+C_p \nu \cdot \left\|\langle u \rangle \right\|_{L_{x_2}^2}^{\frac{2(p+3)}{5-p}} \\
&=& \frac{\nu \left\|\partial_{x_2} \langle u \rangle \right\|_{L_{x_2}^2}^2}{2}+C_p \nu \cdot \left\|\langle u \rangle \right\|_{L_{x_2}^2}^{2+\frac{4(p-1)}{5-p}} \\
&\le& \frac{\nu \left\|\partial_{x_2} \langle u \rangle \right\|_{L_{x_2}^2}^2}{2}+C_p \nu \cdot K^{\frac{4(p-1)}{5-p}} \left\|\langle u \rangle \right\|_{L_{x_2}^2}^{2},
\end{eqnarray*}
where in the last inequality, we use the assumption \eqref{20210518eq01}. 

\medskip

\textit{Estimation of $B$.} Using the Gagliardo–Nirenberg's inequality again, we obtain that
\begin{eqnarray*}
B%
&\le& C_p \nu \cdot \left\|\langle u \rangle \right\|_{L_{x_2}^2} \left\|\nabla u_{\notparallel} \right\|_{L^2}^{p-1} \cdot \left\| u_{\notparallel} \right\|_{L^2} \\
&\le& C_p \nu \left[\left\|\nabla u_{\notparallel} \right\|_{L^2}^{p-1} \cdot \left\| u_{\notparallel} \right\|_{L^2} +\left\|\nabla u_{\notparallel} \right\|_{L^2}^{p-1} \cdot \left\| u_{\notparallel} \right\|_{L^2}  \|\langle u \rangle\|_{L_{x_2}^2}^2 \right]. 
\end{eqnarray*}

\medskip

Combining both estimates of $A$ and $B$, \eqref{20210518eq02} reduces to
\begin{eqnarray} \label{20210518eq05}
&&\frac{1}{2}\frac{d}{dt} \int_\T \left| \langle u \rangle(x_2) \right|^2 dx_2+\frac{\nu}{2} \int_\T \left| \partial_{x_2} \langle u \rangle(x_2) \right|^2 dx_2 \le C_p \nu K^{\frac{4(p-1)}{5-p}} \|\langle u \rangle \|_{L_{x_2}^2}^2 \nonumber \\
&& \quad \quad \quad \quad \quad \quad \quad  +C_p \nu \left[\left\|\nabla u_{\notparallel} \right\|_{L^2}^{p-1} \cdot \left\| u_{\notparallel} \right\|_{L^2} +\left\|\nabla u_{\notparallel} \right\|_{L^2}^{p-1} \cdot \left\| u_{\notparallel} \right\|_{L^2}  \|\langle u \rangle\|_{L_{x_2}^2}^2 \right]. 
\end{eqnarray}

Hence, by the one-dimensional Poincar\'e's inequality and  \eqref{20210518eq03}, we have 
$$
C_p \nu K^{\frac{4(p-1)}{5-p}} \| \langle u \rangle\|_{L_{x_2}^2}^2  \le C_p \nu  \cdot \frac{\lambda_1}{4C_p} \left\| \langle u \rangle \right\|_{L_{x_2}^2}^2 \le \frac{\nu}{4} \cdot \left\| \partial_{x_2} \langle u \rangle \right\|_{L_{x_2}^2}^2,
$$
where we recall that $\lambda_1$ is the smallest eigenvalue of $-\Delta$ on $\T$. Therefore, we can further reduce \eqref{20210518eq05} by 
\begin{eqnarray}  \label{20210518eq53}
&&\frac{1}{2}\frac{d}{dt} \int_\T \left| \langle u \rangle(x_2) \right|^2 dx_2+\frac{\nu}{4} \int_\T \left| \partial_{x_2} \langle u \rangle(x_2) \right|^2 dx_2 \nonumber \\
&& \quad \quad \quad \quad \quad \quad \quad  \le C_p \nu \left[\left\|\nabla u_{\notparallel} \right\|_{L^2}^{p-1} \cdot \left\| u_{\notparallel} \right\|_{L^2} +\left\|\nabla u_{\notparallel} \right\|_{L^2}^{p-1} \cdot \left\| u_{\notparallel} \right\|_{L^2}  \|\langle u \rangle\|_{L_{x_2}^2}^2 \right], 
\end{eqnarray}
which gives the ODE
\begin{equation} \label{20210518eq10}
\frac{1}{2}\frac{d}{dt} \left\| \langle u \rangle \right\|_{L_{x_2}^2}^2 \le C_p \nu \left[\left\|\nabla u_{\notparallel} \right\|_{L^2}^{p-1} \cdot \left\| u_{\notparallel} \right\|_{L^2} +\left\|\nabla u_{\notparallel} \right\|_{L^2}^{p-1} \cdot \left\| u_{\notparallel} \right\|_{L^2}  \|\langle u \rangle\|_{L_{x_2}^2}^2 \right]. 
\end{equation} 
Let 
$$
\rho(t):=\exp \left(-C_p \nu \int_0^t \left\|\nabla u_{\notparallel}(\tau) \right\|_{L^2}^{p-1} \|u_{\notparallel}(\tau) \|_{L^2} d\tau \right)
$$
be the integrating factor, and solving the ODE \eqref{20210518eq10}, we have 
\begin{eqnarray} \label{20210518eq11}
&&\int_\T \left|\langle u \rangle (T, x_2) \right|^2 dx_2 \le \frac{1}{\rho(T)} \int_\T \left|\langle u \rangle(0, x_2) \right|^2 dx_2 \nonumber  \\
&& \quad \quad \quad \quad \quad \quad \quad \quad \quad \quad \quad \quad +\frac{C_p \nu}{\rho(T)} \int_0^T \rho(\tau) \left\| \nabla u_{\notparallel} (\tau) \right\|_{L^2}^{p-1} \left\|u_{\notparallel}(\tau) \right\|_{L^2} d\tau.
\end{eqnarray}
By the bootstrap assumption, we estimate $\frac{1}{\rho(T)}$ as follows:
\begin{eqnarray} \label{20210518eq60}
\frac{1}{\rho(T)}%
&=& \exp\left(C_p \nu \int_0^T \left\|\nabla u_{\notparallel}(\tau) \right\|_{L^2}^{p-1} \left\|u_{\notparallel}(\tau) \right\|_{L^2} d\tau \right) \nonumber \\
&\le& \exp \left(20 C_p \nu \cdot \|u_{\notparallel}(0)\|_{L^2}  \int_0^T \left\|\nabla u_{\notparallel}(\tau) \right\|_{L^2}^{p-1} \cdot e^{-\frac{\lambda_\nu \tau}{4}} d\tau  \right).
\end{eqnarray}
Applying H\"older's inequality, we see that 
\begin{eqnarray*}
&&\int_0^T \left\|\nabla u_{\notparallel}(\tau) \right\|_{L^2}^{p-1}  e^{-\frac{\lambda_\nu \tau}{4}} d\tau \le \left(\int_0^T \left\|\nabla u_{\notparallel}(\tau) \right\|_{L^2}^2 d\tau \right)^{\frac{p-1}{2}} \left( \int_0^T e^{\frac{-\lambda_\nu \tau}{2(3-p)}}d\tau \right)^{\frac{3-p}{2}} \\
&& \quad \quad \quad \quad \quad  =\left( \nu \int_0^T \left\|\nabla u_{\notparallel}(\tau) \right\|_{L^2}^2 d\tau \right)^{\frac{p-1}{2}} \cdot \nu^{\frac{1-p}{2}} \cdot \left( \int_0^T e^{\frac{-\lambda_\nu \tau}{2(3-p)}}d\tau \right)^{\frac{3-p}{2}} \\
&&  \quad \quad \quad \quad \quad \le \nu^{\frac{1-p}{2}} \cdot 10^{\frac{p-1}{2}} \left\|u_{\notparallel}(0)\right\|_{L^2}^{p-1} \cdot \frac{\left[2(3-p) \right]^{\frac{3-p}{2}}}{\lambda_\nu^{\frac{3-p}{2}}}.
\end{eqnarray*}
Therefore, 
\begin{eqnarray*}
\frac{1}{\rho(T)}%
&\le& \exp \left(20 C_p \nu \cdot \|u_{\notparallel}(0)\|_{L^2} \cdot  \nu^{\frac{1-p}{2}} \cdot 10^{\frac{p-1}{2}} \left\|u_{\notparallel}(0)\right\|_{L^2}^{p-1} \cdot \frac{\left[2(3-p) \right]^{\frac{3-p}{2}}}{\lambda_\nu^{\frac{3-p}{2}}}\right)\\
&=& \exp \left(20C_p \cdot 10^{\frac{p-1}{2}} \cdot \left[2(3-p) \right]^{\frac{3-p}{2}} \cdot \left(\frac{\nu}{\lambda_\nu} \right)^{\frac{3-p}{2}} \cdot \left\|u_{\notparallel}(0) \right\|_{L^2}^p \right)\\
&=& \exp \left(C_{p, \nu} \left\|u_{\notparallel}(0) \right\|_{L^2}^p \right). 
\end{eqnarray*}
This estimate together with \eqref{20210518eq11} yields
\begin{eqnarray*}
\left\|\langle u \rangle(T) \right\|_{L_{x_2}^2}^2%
&\le& \exp \left(C_{p, \nu} \left\|u_{\notparallel}(0) \right\|_{L^2}^p \right) \\
&& \quad \quad \quad \quad  \cdot \left[ \left\| \langle u \rangle(0) \right\|_{L_{x_2}^2}^2 +C_p \nu \int_0^T \left\|\nabla u_{\notparallel}(\tau) \right\|_{L^2}^{p-1} \left\|u_{\notparallel}(\tau) \right\|_{L^2} d\tau \right] \\
&& \le \exp\left(C_{p, \nu} \left\|u_{\notparallel}(0) \right\|_{L^2}^p \right) \cdot \left[ \left\|\langle u \rangle(0) \right\|^2_{L_{x_2}^2}+C_{p, \nu} \|u_{\notparallel}(0) \|_{L^2}^p \right].
\end{eqnarray*}
The desired result is proved. 
\end{proof}

We are now ready to bound the two terms in \eqref{20210518eq13} uniformly for all $t \in [0, \widetilde{t}_0]$. Here is the main result.

\begin{prop} \label{20210519prop01}
Assume the bootstrap assumption. Then there exists a $\widetilde{\nu}>0$, which only depends on $p$ and $\|u_{\notparallel}(0)\|_{L^2}$, such that for any $0<\nu<\widetilde{\nu}$ and for any initial data $\langle u \rangle (0, \cdot)$ of \eqref{eqpar} with mean zero and satisfying  
\begin{equation} \label{20210518eq15}
\left\| \langle u \rangle(0, \cdot) \right\|_{L_{x_2}^2} \le \frac{1}{4} \cdot \left( \frac{\lambda_1}{4C_p} \right)^{\frac{5-p}{4(p-1)}}
\end{equation}
the mild solution of \eqref{eqpar} satisfies the following estimates: for any $0 \le t \le \widetilde{t}_0$, 
\begin{enumerate}
    \item [(1).] 
    \begin{equation} \label{20210518eq40}
    \left\|\langle u \rangle(t, \cdot) \right\|_{L_{x_2}^2} \le\frac{1}{2} \cdot \left( \frac{\lambda_1}{4C_p}\right)^{\frac{5-p}{4(p-1)}};
    \end{equation} 
    \item [(2).] 
  \begin{equation} \label{20210518eq401}
    \nu \int_0^t \left\| \partial_{x_2} \langle u \rangle (\tau, \cdot) \right\|_{L_{x_2}^2}^2 d\tau \le 4+4\left( \frac{\lambda_1}{4C_p}\right)^{\frac{5-p}{2(p-1)}}. 
  \end{equation} 
\end{enumerate}
Here, $\lambda_1$ and $C_p$ are the constants defined as in Lemma \ref{20210517prop011}.
\end{prop}

\begin{proof}
(1). Recall from Remark \ref{20210518rem01} that $\lim\limits_{\nu \to 0} C_{p, \nu}=0$, this allows us to choose a $\widetilde{\nu}$ sufficiently small (but only depends on $p$ and $\|u_{\notparallel}(0)\|_{L^2}$, such that for any $0<\nu<\widetilde{\nu}$, we have 
\begin{equation} \label{20210518eq54}
\exp\left(C_{p, \nu} \|u_{\notparallel}(0)\|_{L^2}^p \right)<2 
\end{equation} 
and
$$
C_{p, \nu} \exp\left(C_{p, \nu} \|u_{\notparallel}(0)\|_{L^2}^p \right) \|u_{\notparallel}(0)\|_{L^2}^p< \frac{1}{8} \cdot \left(\frac{\lambda_1}{4C_p} \right)^{\frac{5-p}{2(p-1)}}. 
$$
Let
$$
K:=\frac{1}{2} \cdot \left( \frac{\lambda_1}{4C_p}\right)^{\frac{5-p}{4(p-1)}}
$$
in Lemma \ref{20210517prop011} and $T$ be the maximal time such that $\left\|\langle u \rangle(t) \right\|_{L_{x_2}^2} \le K$ on $[0, T]$. By the continuity of $L^2$ norm of the mild solution, $T>0$. Our goal is to show that $T=\widetilde{t}_0$. Otherwise, assume $0<T<\widetilde{t}_0$. By Lemma \ref{20210517prop011} and \eqref{20210518eq15}, this means
\begin{eqnarray*}
\frac{1}{4} \cdot \left(\frac{\lambda_1}{4C_p} \right)^{\frac{5-p}{2(p-1)}}%
&=& K^2=\left\| \langle u \rangle (T) \right\|_{L_{x_2}^2}^2 \\
&\le&   \exp\left(C_{p, \nu} \|u_{\notparallel}(0)\|_{L^2}^p \right) \cdot \left[\left\| \langle u \rangle(0, \cdot) \right\|_{L_{x_2}^2}^2+C_{p, \nu} \|u_{\notparallel}(0)\|_{L^2}^p \right] \\
&<& \frac{2}{16} \cdot \left(\frac{\lambda_1}{4C_p} \right)^{\frac{5-p}{2(p-1)}}+\frac{1}{8} \cdot \left(\frac{\lambda_1}{4C_p} \right)^{\frac{5-p}{2(p-1)}} \\
&=& \frac{1}{4} \cdot \left(\frac{\lambda_1}{4C_p} \right)^{\frac{5-p}{2(p-1)}},
\end{eqnarray*}
which is a contradiction. This means that \eqref{20210518eq40} has to be true until the maximal time, that is $T=\widetilde{t}_0$. 

\medskip

(2). As a consequence of \eqref{20210518eq40}, we see that the condition \eqref{20210518eq01} in Lemma \ref{20210517prop011} is satisfied, therefore,  \eqref{20210518eq53} is valid, that is, 
\begin{eqnarray*}  
&&\frac{1}{2}\frac{d}{dt} \int_\T \left| \langle u \rangle \right|^2 dx_2+\frac{\nu}{4} \int_\T \left| \partial_{x_2} \langle u \rangle \right|^2 dx_2 \nonumber \\
&& \quad \quad \quad \quad \quad \quad \quad  \le C_p \nu \left[\left\|\nabla u_{\notparallel} \right\|_{L^2}^{p-1} \cdot \left\| u_{\notparallel} \right\|_{L^2} +\left\|\nabla u_{\notparallel} \right\|_{L^2}^{p-1} \cdot \left\| u_{\notparallel} \right\|_{L^2}  \|\langle u \rangle\|_{L_{x_2}^2}^2 \right]. 
\end{eqnarray*}
Integrating the above equation with respect to the interval $[0, t]$, we find that 
\begin{eqnarray*} 
&&\frac{\nu}{4} \int_0^T \left\|\partial_{x_2} \langle u \rangle (\tau) \right\|_{L_{x_2}^2}^2 d\tau  \le \frac{1}{2} \int_{\T} \left| \langle u \rangle (0, x_2) \right|^2 dx_2 \nonumber \\
&& \quad \quad \quad \quad   \quad  \quad \quad \quad \quad \quad \quad +C_p \nu \int_0^t \left\|\nabla u_{\notparallel} (\tau)\right\|_{L^2}^{p-1} \left\|u_{\notparallel} (\tau) \right\|_{L^2} d\tau \nonumber \\
&& \quad \quad \quad \quad  \quad  \quad \quad \quad \quad \quad \quad + C_p \nu \int_0^t \left\|\nabla u_{\notparallel} (\tau)\right\|_{L^2}^{p-1} \left\|u_{\notparallel} (\tau) \right\|_{L^2} \left\|\langle u \rangle (\tau, \cdot) \right\|_{L_{x_2}^2}^2 d\tau. 
\end{eqnarray*}
Now by \eqref{20210518eq15} and \eqref{20210518eq40}, we can further bound the above right hand side by
$$
\left[1+\frac{1}{4} \cdot \left(\frac{\lambda_1}{4C_p} \right)^{\frac{5-p}{2(p-1)}} \right] \cdot C_p \nu \int_0^t \left\|\nabla u_{\notparallel} (\tau)\right\|_{L^2}^{p-1} \left\|u_{\notparallel} (\tau) \right\|_{L^2} d\tau+\frac{1}{32} \cdot \left(\frac{\lambda_1}{4C_p} \right)^{\frac{5-p}{2(p-1)}}. 
$$
To this end, we control the integral in the above term by following exactly the same steps as in \eqref{20210518eq60} to see that 
\begin{eqnarray*}
\frac{\nu}{4} \int_0^T \left\|\partial_{x_2} \langle u \rangle (\tau) \right\|_{L_{x_2}^2}^2 d\tau %
&\le& \left[1+\frac{1}{4} \cdot \left(\frac{\lambda_1}{4C_p} \right)^{\frac{5-p}{2(p-1)}} \right] \cdot C_{p, \nu} \|u_{\notparallel}(0)\|_{L_{x_2}^2}^p \\
&& + \frac{1}{32} \cdot \left(\frac{\lambda_1}{4C_p} \right)^{\frac{5-p}{2(p-1)}}  \\
&\le& \ln 2 \cdot \left[1+\frac{1}{4} \cdot \left(\frac{\lambda_1}{4C_p} \right)^{\frac{5-p}{2(p-1)}} \right] + \frac{1}{32} \cdot \left(\frac{\lambda_1}{4C_p} \right)^{\frac{5-p}{2(p-1)}},
\end{eqnarray*}
which clearly implies the desired estimate \eqref{20210518eq401}. Here in the last estimate, we have used the estimate \eqref{20210518eq54}. The proof is complete. 
\end{proof}

\medskip
\subsection{Bootstrap estimates} The next step is to show that for sufficiently small $\nu$, the bootstrap assumption can be self-improved. More precisely, we will show that under the assumption of Theorem \ref{mainthm03} (or, \eqref{20210518eq15}) and the bootstrap assumption, we have for $\nu$ sufficiently small and $0 \le s \le t \le \widetilde{t}_0$, 
\begin{enumerate}
    \item [(1).] $\|u_{\notparallel}(t)\|_{L^2} \le 15e^{-\frac{\lambda_\nu(t-s)}{4}} \|u_{\notparallel}(s)\|_{L^2}$; 
    
    \medskip
    
    \item [(2).] $\nu \int_s^t \|\nabla u_{\notparallel}(\tau)\|_{L^2}^2 d\tau \le 5 \|u_{\notparallel}(s)\|_{L^2}^2$. 
\end{enumerate}
The above estimates are referred as the \emph{bootstrap estimates}.

\medskip

Let us include some motivation before we proceed. Heuristically, Proposition \ref{20210519prop01} suggests that the $L_{x_2}^2$-norm of $\langle u \rangle (t, \cdot)$ behaves like a ``constant" for large $t$, while $\|u_{\notparallel}(t)\|_{L^2}$ possesses an exponential decay as $t$ increases. This suggests to prove the refined bootstrap estimates, we need to ``eliminate" the contribution of the term $\langle u \rangle$. More precisely, for $u_{\notparallel}$, instead of working directly on the energy estimate of \eqref{eqnotpar}, we modify it a little bit and study the corresponding energy identity:
\begin{eqnarray} \label{modeqnotpar}
&& \partial_t u_{\notparallel}+v_1(x_2) \partial_{x_1} u_{\notparallel}-\nu \Delta u_{\notparallel}= \nu \left( \left| \langle u \rangle+u_{\notparallel} \right|^p-\left| \langle u \rangle \right|^p \right) \nonumber \\
&& \quad \quad \quad \quad \quad \quad \quad \quad \quad \quad \quad \quad \quad \quad \quad \quad -\nu \int_{\T} \left( \left| \langle u \rangle+u_{\notparallel} \right|^p-\left| \langle u \rangle \right|^p \right) dx_1.
\end{eqnarray}

We need the following simple lemma, which plays an important role to ``cancel" the contribution $\langle u \rangle$ in our coming proof. 

\begin{lem} \label{20210520lem10} 
For any $a, b \in \R$ and $p \ge 1$, 
\begin{equation} \label{20210522eq02}
\left|(a+b)^p-b^p \right| \le (p-1) \max\left\{1, 2^{p-2} \right\} \cdot |a| \cdot \left(|a|^{p-1}+|b|^{p-1} \right).
\end{equation} 
\end{lem}

\begin{proof}
The proof follows clearly by the mean value theorem and the elementary estimate  
$$
(|a|+|b|)^{p-1} \le (p-1) \max \left\{1, 2^{p-2} \right\} \left(|a|^{p-1}+|b|^{p-1} \right),
$$
and hence we omit it here. 
\end{proof}

We are now ready to prove the second estimate in the bootstrap estimate. 

\begin{prop} \label{best02prop}
Assume the bootstrap assumption, $u_0\in L^2_0$, \eqref{20210518eq14b} and \eqref{20210518eq15}. Then there exists a $\nu_1>0$ which only depends on $p$ and $\|u_{\notparallel}(0)\|_{L^2}$, such that for any $0<\nu<\nu_1$ and any $0 \le s \le t \le \widetilde{t}_0$, 
\begin{equation} \label{best02}
\nu \int_s^t \|\nabla u_{\notparallel}(\tau)\|_{L^2}^2 d\tau \le 5 \|u_{\notparallel}(s)\|_{L^2}^2.
\end{equation} 
\end{prop}

\begin{proof}
We begin with commenting that the constant $C_p'$ in the proof of this lemma might change from line by line, which, however, will only depend on $p$. By the energy identity corresponding to \eqref{modeqnotpar}, we have
\begin{eqnarray} \label{20210520eq01} 
&& \frac{1}{2} \frac{d}{dt} \|u_{\notparallel}\|_{L^2}^2+\nu \left\|\nabla u_{\notparallel} \right\|_{L^2}^2=\nu \int_{\T^2} u_{\notparallel}  \left( \left| \langle u \rangle+u_{\notparallel} \right|^p-\left| \langle u \rangle \right|^p \right) dx_1dx_2 \nonumber\\
&& \quad \quad \quad \quad \le \nu \int_{\T^2} \left|u_{\notparallel}\right| \left| \left| \langle u \rangle+u_{\notparallel} \right|^p-\left| \langle u \rangle \right|^p \right| dx_1dx_2 \nonumber\\
&& \quad \quad \quad \quad \le C_p' \nu \int_{\T^2} \left|u_{\notparallel}\right|^2 \left(\left|\langle u \rangle\right|^{p-1}+\left|u_{\notparallel}\right|^{p-1} \right) dx_1dx_2 \nonumber \\
&& \quad \quad \quad \quad \le C_p' \nu \int_{\T^2} \left|u_{\notparallel} \right|^{p+1} dx_1dx_2+C_p' \nu \int_{\T^2} |u_{\notparallel}|^2 \left| \langle u \rangle \right|^{p-1} dx_1dx_2, 
\end{eqnarray}
where in the first estimate, we have used the fact that
$$
\int_\T u_{\notparallel}(t, x_1, x_2)dx_1=0. 
$$
For the second integral in the \eqref{20210520eq01}, we further bound it as 
\begin{eqnarray*}
\int_{\T^2} \left|u_{\notparallel} \right|^2 \left| \langle u \rangle \right|^{p-1} dx_1dx_2 %
&=& \int_\T \left[ \left|\langle u \rangle \right|^{p-1} \cdot \int_\T  \left|u_{\notparallel} \right|^2 dx_1 \right] dx_2 \\
&\le& \left( \int_\T \left| \langle u \rangle \right|^{p+1} dx_2 \right)^{\frac{p-1}{p+1}} \cdot \left( \int_\T \left| \int_\T |u_{\notparallel}|^2 dx_1 \right|^{\frac{p+1}{2}} dx_2 \right)^{\frac{2}{p+1}} \\
&\le& \|\langle u \rangle\|_{L_{x_2}^{p+1}}^{p-1} \cdot \|u_{\notparallel} \|_{L^{p+1}}^2, 
\end{eqnarray*}
where in the last estimate follows from the Minkowski's inequality. Therefore, combining the above estimates yields
\begin{equation} \label{20210520eq100}
\frac{1}{2} \frac{d}{dt} \|u_{\notparallel}\|_{L^2}^2+\nu \left\|\nabla u_{\notparallel} \right\|_{L^2}^2 \le C_p' \nu \left\|u_{\notparallel} \right\|_{L^{p+1}}^2 \left[ \left\|\langle u \rangle \right\|_{L_{x_2}^{p+1}}^{p-1}+\left\|u_{\notparallel}\right\|_{L^{p+1}}^{p-1} \right]. 
\end{equation} 
Integrating the above estimate with respect to the interval $[s, t]$, we get
\begin{eqnarray} \label{20210520eq02}
\nu \int_s^t \left\|\nabla u_{\notparallel}(\tau) \right\|_{L^2}^2 d\tau%
&\le& \frac{\|u_{\notparallel}(s)\|_{L^2}^2}{2}+ C_p'\nu \int_s^t\left\|u_{\notparallel} (\tau)\right\|_{L^{p+1}}^{p+1} d\tau  \nonumber \\
&& \quad \quad \quad \quad \quad \quad  + C_p' \nu \int_s^t \left\|u_{\notparallel} (\tau)\right\|_{L^{p+1}}^2 \left\|\langle u \rangle (\tau) \right\|_{L_{x_2}^{p+1}}^{p-1}d\tau.
\end{eqnarray}
Now we use Gagliardo–Nirenberg's inequality, the bootstrap assumption and Proposition \ref{20210519prop01} to estimate the two integrals in \eqref{20210520eq02} . More precisely, for the first integral
\begin{eqnarray} \label{20210520eq20} 
&&\int_s^t\left\|u_{\notparallel} (\tau)\right\|_{L^{p+1}}^{p+1} d\tau \le \int_s^t \left( \left\|\nabla u_{\notparallel}(\tau) \right\|_{L^2}^{\frac{p-1}{p+1}} \cdot \left\|u_{\notparallel}(\tau) \right\|_{L^2}^{\frac{2}{p+1}} \right)^{p+1} d\tau \nonumber \\
&& \quad \quad \quad \quad = C_p' \|u_{\notparallel}(s)\|_{L^2}^2 \cdot \int_s^t e^{-\frac{\lambda_\nu(\tau-s)}{2}} \cdot  \left\|\nabla u_{\notparallel}(\tau) \right\|_{L^2}^{p-1} d\tau \nonumber\\
&&\quad \quad \quad \quad \le C_p'\|u_{\notparallel} (s)\|_{L^2}^2 \cdot \left(\int_s^t e^{-\frac{\lambda_\nu(\tau-s)}{3-p}} d\tau \right)^{\frac{3-p}{2}} \nonumber \\
&& \quad \quad \quad \quad \quad \quad \quad \quad \quad  \quad \quad \quad \quad \cdot \nu^{\frac{1-p}{2}} \cdot \left( \nu \int_s^t \left\| \nabla u_{\notparallel}(\tau) \right\|_{L^2}^2 d\tau \right)^{\frac{p-1}{2}} \nonumber\\
&& \quad \quad \quad \quad \le  \frac{C_p'\nu^{\frac{1-p}{2}}}{\lambda_\nu^{\frac{3-p}{2}}} \cdot \|u_{\notparallel}(s)\|_{L^2}^{p-1}  \|u_{\notparallel}(s)\|_{L^2}^2  \le \frac{C_p' \nu^{\frac{1-p}{2}}}{\lambda_\nu^{\frac{3-p}{2}}} \cdot  \|u_{\notparallel}(0)\|_{L^2}^{p-1}  \|u_{\notparallel}(s)\|_{L^2}^2,
\end{eqnarray} 
where in the last estimate, we have used the first estimate of the bootstrap assumption on the interval $[0, s]$. Next we estimate the second integral in \eqref{20210520eq02}. To begin with, we let 
\begin{equation} \label{20210520eq40}
\left(\frac{p+1}{p-1}, \frac{4(p+1)}{(p-1)^2}, D_p \right)
\end{equation} 
be a H\"older-triple, that is $\frac{p-1}{p+1}+\frac{(p-1)^2}{4(p+1)}+\frac{1}{D_p}=1$. 
Here, the value of $D_p$, which is $\frac{4(p+1)}{8-(p-1)^2}$ but it is not particularly important, and all we need to know is that $D_p>1$. Now for the second integral, we have
\begin{eqnarray*}
&&\int_s^t \left\|u_{\notparallel} (\tau)\right\|_{L^{p+1}}^2 \left\|\langle u \rangle (\tau) \right\|_{L_{x_2}^{p+1}}^{p-1}d\tau \\
&& \quad \quad \le \int_s^t \left(\left\|\nabla u_{\notparallel}(\tau) \right\|_{L^2}^{\frac{p-1}{p+1}} \left\|u_{\notparallel}(\tau) \right\|_{L^2}^{\frac{2}{p+1}} \right)^2 \\
&& \quad \quad \quad \quad \quad \quad \quad \quad \quad \quad \quad \cdot \left(\left\|\partial_{x_2} \langle u \rangle(\tau) \right\|_{L_{x_2}^2}^{\frac{p-1}{2(p+1)}} \left\| \langle u \rangle(\tau) \right\|_{L_{x_2}^2}^{\frac{p+3}{2(p+1)}} \right)^{p-1} d\tau \\
&& \quad \quad=\int_s^t \left\|\nabla u_{\notparallel}(\tau) \right\|_{L^2}^{\frac{2(p-1)}{p+1}} \left\|u_{\notparallel}(\tau) \right\|_{L^2}^{\frac{4}{p+1}}  \left\|\partial_{x_2} \langle u \rangle(\tau) \right\|_{L_{x_2}^2}^{\frac{(p-1)^2}{2(p+1)}} \left\| \langle u \rangle(\tau) \right\|_{L_{x_2}^2}^{\frac{(p+3)(p-1)}{2(p+1)}}  d\tau \\
&& \quad \quad \le C_p' \int_s^t \left\|\nabla u_{\notparallel}(\tau) \right\|_{L^2}^{\frac{2(p-1)}{p+1}} \left\|u_{\notparallel}(\tau) \right\|_{L^2}^{\frac{4}{p+1}}  \left\|\partial_{x_2} \langle u \rangle(\tau) \right\|_{L_{x_2}^2}^{\frac{(p-1)^2}{2(p+1)}}  d\tau \\
&& \quad \quad \le C_p' \cdot \left\|u_{\notparallel}(s) \right\|_{L^2}^{\frac{4}{p+1}} \cdot \int_s^t \left\|\nabla u_{\notparallel}(\tau) \right\|_{L^2}^{\frac{2(p-1)}{p+1}} \left\|\partial_{x_2} \langle u \rangle(\tau) \right\|_{L_{x_2}^2}^{\frac{(p-1)^2}{2(p+1)}}  e^{-\frac{\lambda_\nu (\tau-s)}{p+1}} d\tau. 
\end{eqnarray*} 
Applying the H\"older's inequality with the triple in \eqref{20210520eq40}, we bound the integral in the last line above by 
$$
\left(\int_s^t \left\| \nabla u_{\notparallel}(\tau) \right\|_{L^2}^2 d\tau \right)^{\frac{p-1}{p+1}} \cdot \left( \int_s^t \left\| \partial_{x_2} \langle u \rangle (\tau) \right\|_{L_{x_2}^2}^2 d\tau \right)^{\frac{(p-1)^2}{4(p+1)}} \cdot \left(\int_s^t e^{-\frac{\lambda_\nu D_p(\tau-s)}{p+1}} d\tau \right)^{\frac{1}{D_p}},
$$
which, by the bootstrap assumption and Proposition \ref{20210519prop01}, can be further controlled by 
\begin{eqnarray*}
&& \frac{C_p' \nu^{-1+\frac{1}{D_p}}}{\lambda_{\nu}^{\frac{1}{D_p}}} \left( \nu \int_s^t \left\|\nabla u_{\notparallel}(\tau) \right\|_{L^2}^2 d\tau \right)^{\frac{p-1}{p+1}} \cdot  \left( \nu \int_s^t \left\| \partial_{x_2} \langle u \rangle (\tau) \right\|_{L_{x_2}^2}^2 d\tau \right)^{\frac{(p-1)^2}{4(p+1)}} \\
&& \quad \quad \quad \quad \quad   \le \frac{C_p' \nu^{-1+\frac{1}{D_p}}}{\lambda_{\nu}^{\frac{1}{D_p}}} \cdot \left\|u_{\notparallel}(s) \right\|_{L^2}^{\frac{2(p-1)}{p+1}},
\end{eqnarray*}
which implies 
\begin{equation} \label{20210520eq21} 
\int_s^t \left\|u_{\notparallel} (\tau)\right\|_{L^{p+1}}^2 \left\|\langle u \rangle (\tau) \right\|_{L_{x_2}^{p+1}}^{p-1}d\tau \le \frac{C_p' \nu^{-1+\frac{1}{D_p}}}{\lambda_{\nu}^{\frac{1}{D_p}}} \|u_{\notparallel}(s)\|_{L^2}^2.
\end{equation} 
Finally, combining \eqref{20210520eq02}, \eqref{20210520eq20} and \eqref{20210520eq21}, we get
\begin{eqnarray*}
&& \nu \int_s^t \left\|\nabla u_{\notparallel}(\tau) \right\|_{L^2}^2 d\tau \le \bigg[\frac{1}{2}+2C_p' \left(\frac{\nu}{\lambda_\nu} \right)^{\frac{3-p}{2}} \|u_{\notparallel}(0)\|_{L^2}^{p-1} \\
&& \quad \quad \quad \quad \quad \quad \quad \quad \quad \quad \quad \quad  \quad \quad \quad \quad  +2C_p' \left(\frac{\nu}{\lambda_\nu} \right)^{\frac{1}{D_p}} \bigg] \left\|u_{\notparallel}(s) \right\|^2_{L^2}.
\end{eqnarray*}
This clearly implies the desired estimate \eqref{best02} as long as we choose  $\nu \in \left(0, \nu_1\right)$, where $\nu_1$ satisfies $0<\nu_1<\widetilde{\nu}$ with $\widetilde{\nu}$ being defined in Proposition \ref{20210519prop01} and 
$$
2C_p' \left(\frac{\nu_1}{\lambda_{\nu_1}} \right)^{\frac{3-p}{2}} \|u_{\notparallel}(0)\|_{L^2}^{p-1}+2C_p' \left(\frac{\nu_1}{\lambda_{\nu_1}} \right)^{\frac{1}{D_p}} \le 4.
$$
Note that here we have implicitly used the fact that $\frac{\nu}{\lambda_\nu} \to 0$ as $\nu \to 0$.
\end{proof}

Now we turn to the proof of the first inequality in the bootstrap estimates. The strategy to prove such an improved estimate is to apply iteration. We need the following two lemmas. 

\begin{lem} \label{bestlem01}
Assume the bootstrap assumption, $u_0\in L^2_0$, \eqref{20210518eq14b} and \eqref{20210518eq15}. Then there exists a $\nu_2>0$, which only depends on $p$ and $\|u_{\notparallel}(0)\|_{L^2}$, such that for any $0<\nu<\nu_2$ and any $0 \le s \le t \le \widetilde{t}_0$, 
\begin{equation} \label{20210521eq101}
\left\|u_{\notparallel}(t) \right\|_{L^2} \le \frac{3}{2}  \|u_{\notparallel}(s)\|_{L^2}. 
\end{equation} 
\end{lem}

\begin{proof}
The proof of this lemma is immediate from the proof of Proposition \ref{best02prop}. Indeed, the estimate \eqref{20210520eq100} also implies 
\begin{eqnarray*}
\frac{\|u_{\notparallel}(t)\|_{L^2}^2}{2}%
&\le& \frac{\|u_{\notparallel}(s)\|_{L^2}^2}{2}+ 2C_p'\nu \int_s^t\left\|u_{\notparallel} (\tau)\right\|_{L^{p+1}}^{p+1} d\tau  \nonumber \\
&& \quad \quad \quad \quad \quad \quad  + 2C_p' \nu \int_s^t \left\|u_{\notparallel} (\tau)\right\|_{L^{p+1}}^2 \left\|\langle u \rangle (\tau) \right\|_{L_{x_2}^{p+1}}^{p-1}d\tau.
\end{eqnarray*}
The desired claim then follows exactly by the same approach to control the right hand side of the estimate \eqref{20210520eq02}, with a proper choice of $\nu_2$. 
\end{proof}

\begin{lem} \label{best2lem02}
Assume the bootstrap assumption, $u_0\in L^2_0$, \eqref{20210518eq14b} and \eqref{20210518eq15}. Let $\tau^*:=\frac{4}{\lambda_\nu}$. If $\widetilde{t}_0 \ge \tau^*$, then there exists a $\nu_3>0$, which only depends on $p$ and $\|u_{\notparallel}(0)\|_{L^2}$, such that any $s \in \left[0, \widetilde{t}_0-\tau^* \right]$ and any $0<\nu<\nu_3$, one has
\begin{equation} \label{20210521eq102}
\left\|u_{\notparallel}(\tau^*+s) \right\|_{L^2} \le \frac{1}{e} \left\|u_{\notparallel}(s) \right\|_{L^2}. 
\end{equation}
Here we recall that $\widetilde{t}_0$ is defined in the bootstrap assumption, which is the maximal time such that the two estimates in bootstrap assumption hold. 
\end{lem}

\begin{proof}
Again, we start with commenting that the constant $C_p''$ that we will use in the proof of this lemma might change from line by line, but it will only depend on $p$.

Note that the major difference between the estimates \eqref{20210521eq101} and \eqref{20210521eq102} is that the coefficient in the later estimate is strictly smaller than $1$, and hence the proof of Lemma \ref{bestlem01} does not work here. The idea is to use the mild solution of $u_{\notparallel}$. For the same reason as in \eqref{modeqnotpar}, we write the mild solution \eqref{20210512eq201} of $u_{\notparallel}$ as follows: 
\begin{eqnarray*}
&& u_{\notparallel}(\tau^*+s)= \calS_{\tau^*} \left(u_{\notparallel}(s) \right)+\nu \int_s^{\tau^*+s} \calS_{\tau^*+s-\tau} \bigg( |\langle u \rangle(\tau)+u_{\notparallel}(\tau)|^p-\left| \langle u \rangle(\tau) \right|^p \\
&& \quad \quad \quad \quad \quad \quad  \quad \quad \quad \quad \quad \quad \quad \quad -\int_\T  \left[|\langle u \rangle(\tau)+u_{\notparallel}(\tau)|^p-\left|\langle u \rangle(\tau) \right|^p\right]  dx_1 \bigg) d\tau. 
\end{eqnarray*}
Taking $L^2$-norm on both sides of the above estimate, and by \eqref{20210512eq01} and  Lemma \ref{20210520lem10}, we see that
\begin{eqnarray} \label{20210521eq30}
&&\left\|u_{\notparallel}(\tau^*+s) \right\|_{L^2}\le  \left\|\calS_{\tau^*} \left(u_{\notparallel}(s) \right) \right\|_{L^2}+ \nu \int_s^{\tau^*+s} \bigg\| |\langle u \rangle(\tau)+u_{\notparallel}(\tau)|^p-\left| \langle u \rangle(\tau) \right|^p \nonumber \\
&& \quad \quad \quad \quad \quad \quad  \quad \quad \quad \quad \quad \quad \quad \quad -\int_\T  \left[|\langle u \rangle(\tau)+u_{\notparallel}(\tau)|^p-\left|\langle u \rangle(\tau) \right|^p \right] dx_1 \bigg\|_{L^2} d\tau \nonumber\\
&& \quad \le  \frac{10}{e^4} \cdot \|u_{\notparallel}(s)\|_{L^2} + \nu \int_s^{\tau^*+s} \bigg\| |\langle u \rangle(\tau)+u_{\notparallel}(\tau)|^p-\left| \langle u \rangle(\tau) \right|^p \nonumber\\
&& \quad \quad \quad \quad \quad \quad  \quad \quad \quad \quad \quad \quad \quad \quad -\int_\T  \left[|\langle u \rangle(\tau)+u_{\notparallel}(\tau)|^p-\left|\langle u \rangle(\tau) \right|^p\right]  dx_1 \bigg\|_{L^2} d\tau, \nonumber \\
&& \quad\le  \frac{10}{e^4} \cdot \|u_{\notparallel}(s)\|_{L^2}+2\nu \int_s^{\tau^*+s} \left\| |\langle u \rangle(\tau)+u_{\notparallel}(\tau)|^p-\left| \langle u \rangle(\tau) \right|^p \right\|_{L^2} d\tau \nonumber\\
&& \quad  \le \frac{10}{e^4} \cdot \|u_{\notparallel}(s)\|_{L^2}+C_p''\nu \int_s^{\tau^*+s} \left\| |u_{\notparallel}(\tau)| \left(\left| \langle u \rangle (\tau) \right|^{p-1}+\left|u_{\notparallel}(\tau) \right|^{p-1} \right) \right\|_{L^2} d\tau \nonumber\\
&& \quad \le \frac{10}{e^4} \cdot \|u_{\notparallel}(s)\|_{L^2}+ C_p'' \nu \int_s^{\tau^*+s} \left\| |u_{\notparallel}(\tau)|^p \right\|_{L^2} d\tau \nonumber\\ 
&& \quad \quad \quad \quad \quad \quad \quad \quad  \quad \quad \quad \quad \quad \quad\quad +C_p'' \nu \int_s^{\tau^*+s} \left\| |u_{\notparallel}(\tau) \left| \langle u \rangle (\tau) \right|^{p-1} \right\|_{L^2} d\tau. 
\end{eqnarray}
Next we control the last two integrals in \eqref{20210521eq30} by following a similar approach as in the estimate \eqref{20210520eq02}. Note that in \eqref{20210520eq02}, the Gagliardo–Nirenberg's inequality is applied to control the $L^{p+1}$-norm, while here for $L^{2p}$-norm. We now turn to the details. We estimate the first integral as follows:
\begin{eqnarray} \label{20210521eq41} 
&&\int_s^{\tau^*+s} \left\|u_{\notparallel} (\tau)\right\|_{L^{2p}}^{p} d\tau \le \int_s^{\tau^*+s} \left( \left\|\nabla u_{\notparallel}(\tau) \right\|_{L^2}^{\frac{p-1}{p}} \cdot \left\|u_{\notparallel}(\tau) \right\|_{L^2}^{\frac{1}{p}} \right)^p d\tau \nonumber \\
&& \quad \quad \quad \quad = \int_s^t \left\|\nabla u_{\notparallel}(\tau) \right\|_{L^2}^{p-1} \cdot \left\|u_{\notparallel}(\tau) \right\|_{L^2} d\tau \nonumber\\
&& \quad \quad \quad \quad = C_p'' \|u_{\notparallel}(s)\|_{L^2} \cdot \int_s^{\tau^*+s} e^{-\frac{\lambda_\nu(\tau-s)}{4}} \cdot  \left\|\nabla u_{\notparallel}(\tau) \right\|_{L^2}^{p-1} d\tau \nonumber\\
&&\quad \quad \quad \quad \le C_p''\|u_{\notparallel} (s)\|_{L^2}\cdot \left(\int_s^{\tau^*+s} e^{-\frac{\lambda_\nu(\tau-s)}{6-2p}} d\tau \right)^{\frac{3-p}{2}} \nonumber \\
&& \quad \quad \quad \quad \quad \quad \quad \quad \quad  \quad \quad \quad \quad \cdot \nu^{\frac{1-p}{2}} \cdot \left( \nu \int_s^{\tau^*+s} \left\| \nabla u_{\notparallel}(\tau) \right\|_{L^2}^2 d\tau \right)^{\frac{p-1}{2}} \nonumber\\
&& \quad \quad \quad \quad \le C_p'' \cdot \frac{\nu^{\frac{1-p}{2}}}{\lambda_\nu^{\frac{3-p}{2}}} \cdot  \|u_{\notparallel}(s)\|_{L^2}^{p-1} \cdot \|u_{\notparallel}(s)\|_{L^2}\nonumber \\
&& \quad \quad \quad \quad \le C_p'' \cdot \frac{\nu^{\frac{1-p}{2}}}{\lambda_\nu^{\frac{3-p}{2}}} \cdot  \|u_{\notparallel}(0)\|_{L^2}^{p-1} \cdot \|u_{\notparallel}(s)\|_{L^2}. 
\end{eqnarray} 
While for the second one, we have 
\begin{eqnarray} \label{20210521eq42}
&& \int_s^{\tau^*+s} \left\| |u_{\notparallel}(\tau) \left| \langle u \rangle (\tau) \right|^{p-1} \right\|_{L^2} d\tau \nonumber \\
&& \quad \quad =\int_s^{\tau^*+s} \left[ \int_{\T^2} |u_{\notparallel}(\tau)|^2 \left| \langle u \rangle (\tau) \right|^{2(p-1)} dx_1dx_2 \right]^{\frac{1}{2}} d\tau \nonumber \\
&& \quad \quad = \int_s^{\tau^*+s} \left[ \int_\T \left| \langle u \rangle (\tau) \right|^{2(p-1)} \left(\int_\T |u_{\notparallel}(\tau)|^2 dx_1 \right) dx_2 \right]^{\frac{1}{2}} d\tau \nonumber \\
&& \quad \quad \le \int_s^{\tau^*+s} \left[ \int_\T \left| \langle u \rangle (\tau) \right|^{2(p-1)} \left(\int_\T |u_{\notparallel}(\tau)|^{2p} dx_1 \right)^{\frac{1}{p}} dx_2 \right]^{\frac{1}{2}} d\tau \nonumber \\
&& \quad \quad \le \int_s^{\tau^*+s} \left[ \int_\T \left|\langle u \rangle(\tau) \right|^{2p} dx_2 \right]^{\frac{p-1}{2p}} \cdot \left[ \int_{\T^2} |u_{\notparallel}(\tau)|^{2p} dx_1dx_2 \right]^{\frac{1}{2p}} d\tau \nonumber \\
&& \quad \quad =\int_s^{\tau^*+s} \left\| \langle u \rangle (\tau) \right\|_{L_{x_2}^{2p}}^{p-1} \cdot \left\|u_{\notparallel}(\tau) \right\|_{L^{2p}} d\tau \nonumber \\
&& \quad \quad \le C_p'' \int_s^{\tau^*+s} \left\|\partial_{x_2} \langle u \rangle (\tau) \right\|_{L_{x_2}^2}^{\frac{(p-1)^2}{2p}} \cdot \left\|\langle u \rangle (\tau) \right\|_{L_{x_2}^2}^{\frac{p^2-1}{2p}} \nonumber \\
&& \quad \quad \quad \quad \quad \quad \quad \quad \quad \quad \quad \quad \quad \cdot \|\nabla u_{\notparallel}(\tau) \|_{L^2}^{\frac{p-1}{p}} \cdot \|u_{\notparallel}(\tau) \|_{L^2}^{\frac{1}{p}} d\tau \nonumber \\
&& \quad \quad \le C_p'' \int_s^{\tau^*+s} \left\|\partial_{x_2} \langle u \rangle (\tau) \right\|_{L_{x_2}^2}^{\frac{(p-1)^2}{2p}} \cdot  \|\nabla u_{\notparallel}(\tau) \|_{L^2}^{\frac{p-1}{p}} \cdot \|u_{\notparallel}(\tau) \|_{L^2}^{\frac{1}{p}} d\tau \nonumber \\
&& \quad \quad \le C_p'' \|u_{\notparallel}(s) \|_{L^2}^{\frac{1}{p}}  \int_s^{\tau^*+s} e^{-\frac{\lambda_\nu(\tau-s)}{4p}} \cdot \left\|\partial_{x_2} \langle u \rangle (\tau) \right\|_{L_{x_2}^2}^{\frac{(p-1)^2}{2p}} \cdot  \|\nabla u_{\notparallel}(\tau) \|_{L^2}^{\frac{p-1}{p}} d\tau.
\end{eqnarray}
Let
$$
\left( \frac{4p}{(p-1)^2}, \frac{2p}{p-1}, E_p \right)
$$
be a H\"older-triple, where $E_p=\frac{4p}{2p+2-(p-1)^2}>1$. Then 
\begin{eqnarray} \label{20210521eq43}
\eqref{20210521eq42}%
&\le& C_p'' \|u_{\notparallel}(s)\|_{L^2}^\frac{1}{p} \cdot \left(\int_s^{\tau^*+s} \left\|\partial_{x_2} \langle u \rangle (\tau) \right\|^2_{L_{x_2}^2} d\tau \right)^{\frac{(p-1)^2}{4p}} \nonumber \\
&& \quad \quad  \cdot \left(\int_s^{\tau^*+s} \left\|\nabla u_{\notparallel}(\tau) \right\|_{L^2}^2 d\tau \right)^{\frac{p-1}{2p}} \cdot \left( \int_s^{\tau^*+s} e^{-\frac{\lambda_\nu E_p(\tau-s)}{4p}} d\tau \right)^{\frac{1}{E_p}} \nonumber \\
&\le& C_p'' \cdot \nu^{-1+\frac{1}{E_p}} \cdot  \|u_{\notparallel}(s)\|_{L^2}^\frac{1}{p} \cdot \left(\nu \int_s^{\tau^*+s} \left\|\partial_{x_2} \langle u \rangle (\tau) \right\|^2_{L_{x_2}^2} d\tau \right)^{\frac{(p-1)^2}{4p}} \nonumber \\
&& \quad \quad  \cdot \left(\nu \int_s^{\tau^*+s} \left\|\nabla u_{\notparallel}(\tau) \right\|_{L^2}^2 d\tau \right)^{\frac{p-1}{2p}} \cdot \frac{1}{\lambda_\nu^{\frac{1}{E_p}}} \nonumber \\
&\le& C_p'' \cdot \frac{\nu^{-1+\frac{1}{E_p}}}{\lambda_\nu^{\frac{1}{E_p}}} \cdot \|u_{\notparallel}(s)\|_{L^2}^{\frac{1}{p}} \cdot \|u_{\notparallel}(s)\|_{L^2}^{\frac{p-1}{p}} = C_p'' \cdot \frac{\nu^{-1+\frac{1}{E_p}}}{\lambda_\nu^{\frac{1}{E_p}}} \cdot \|u_{\notparallel}(s)\|_{L^2}. 
\end{eqnarray}
Finally, combining \eqref{20210521eq30}, \eqref{20210521eq41}, \eqref{20210521eq42} and \eqref{20210521eq43}, we get 
\begin{eqnarray*}
&& \left\|u_{\notparallel}(\tau^*+s) \right\|_{L^2} \le \frac{10}{e^4} \cdot \|u_{\notparallel}(s) \|_{L^2}+C_p'' \left(\frac{\nu}{\lambda_\nu} \right)^{\frac{3-p}{2}} \cdot \|u_{\notparallel}(0)\|_{L^2}^{p-1} \cdot \|u_{\notparallel}(s)\|_{L^2} \\
&& \quad \quad \quad \quad \quad \quad \quad \quad \quad \quad \quad \quad \quad \quad \quad+ C_p'' \left(\frac{\nu}{\lambda_{\nu}} \right)^{\frac{1}{E_p}} \|u_{\notparallel}(s)\|_{L^2}. 
\end{eqnarray*} 
It is clear that the desired estimate \eqref{20210521eq102} follows from the above estimate, as long as we take $\nu \in (0, \nu_3)$, where $0<\nu_3<\widetilde{\nu}$, where $\widetilde{\nu}$ is defined in Proposition \ref{20210519prop01}, and 
$$
 C_p'' \left(\frac{\nu}{\lambda_\nu} \right)^{\frac{3-p}{2}} \cdot \|u_{\notparallel}(0)\|_{L^2}^{p-1}+C_p''\left(\frac{\nu}{\lambda_{\nu}} \right)^{\frac{1}{E_p}}<\frac{1}{100}. 
$$
The proof is complete. 
\end{proof}

We are now ready to prove the first inequality in the bootstrap estimates. 

\begin{prop} \label{best03prop}
Assume the bootstrap assumption, $u_0\in L^2_0$, \eqref{20210518eq14b} and \eqref{20210518eq15}. Then there exists a $\nu^*>0$ which only depends on $p$ and $\|u_{\notparallel}(0)\|_{L^2}$, such that for any $0<\nu<\nu^*$ and any $0 \le s \le t \le \widetilde{t}_0$, 
\begin{equation} \label{best03}
\|u_{\notparallel}(t)\|_{L^2} \le 15e^{-\frac{\lambda_\nu(t-s)}{4}} \|u_{\notparallel}(s)\|_{L^2}.
\end{equation} 
\end{prop}

\begin{proof}
Let $0<\nu^*<\min\{\widetilde{\nu}, \nu_1, \nu_2, \nu_3\}$, so that all the restrictions in Proposition \ref{20210519prop01}, Proposition \ref{best02prop}, Lemma \ref{bestlem01} and Lemma \ref{best2lem02} hold. If $\widetilde{t}_0<\tau^*$, then we have $t-s \le \tau^*=\frac{4}{\lambda_\nu}$, and hence by Lemma \ref{bestlem01}, 
$$
\|u_{\notparallel}(t)\|_{L^2} \le \frac{3}{2} \|u_{\notparallel}(s)\|_{L^2} \le \frac{15}{e} \|u_{\notparallel}(s)\|_{L^2} \le 15e^{-\frac{\lambda_\nu(t-s)}{4}} \|u_{\notparallel}(s)\|_{L^2}, 
$$
which implies the desired estimate \eqref{best03}. If $\widetilde{t}_0 \ge \tau^*$, then for any $0 \le s \le t \le \widetilde{t}_0$, we can find some $n \in \Z_+$, such that $t \in [n\tau^*+s, (n+1)\tau^*+s)$. Then by Lemma \ref{bestlem01} and Lemma \ref{best2lem02}, we have
\begin{eqnarray*}
\|u_{\notparallel}(t)\|_{L^2}%
& \le& \frac{3}{2} \cdot \|u_{\notparallel}(n\tau^*+s) \|_{L^2} \le \frac{3}{2e^n} \cdot \|u_{\notparallel}(s)\|_{L^2} \\
&\le& \frac{3}{2} \cdot e^{1-\frac{(t-s)}{\tau^*}} \cdot \|u_{\notparallel}(s)\|_{L^2} \le 15e^{-\frac{\lambda_\nu(t-s)}{4}} \|u_{\notparallel}(s)\|_{L^2}.
\end{eqnarray*}
This concludes the proof of the proposition. 
\end{proof}

\medskip

\subsection{Proof of Theorem \ref{mainthm03}} Note that it suffices to show the maximal time $\widetilde{t}_0$ has to be infinity. The regularity of the solution $u$ follows clearly from the bootstrap assumption and Proposition \ref{20210519prop01}. Assume that $\widetilde{t}_0<\infty$. Then on one hand, the bootstrap estimates hold on the interval $[0, \widetilde{t}_0]$. On the other hand, by the first inequality in the bootstrap estimate and Proposition \ref{20210519prop01}, we have for any $0 \le t \le \widetilde{t}_0$
$$
\left\|u \left(t, \cdot \right) \right\|_{L^2}<\infty, 
$$
which, by Corollary \ref{cor01}, implies that $\widetilde{t}_0$ is \emph{not} the maximal time of existence of $u$. Therefore, by continuity and the bootstrap estimate, we can find a sufficiently small $\varepsilon>0$ (which may depend on the continuity at $\widetilde{t}_0$), such that for any $0 \le s \le t \le \widetilde{t}_0+\varepsilon$,
\begin{enumerate}
    \item [(1).] $\|u_{\notparallel}(t)\|_{L^2} \le 18e^{-\frac{\lambda_\nu(t-s)}{4}} \|u_{\notparallel}(s)\|_{L^2}$; 
    
    \medskip
    
    \item [(2).] $\nu \int_s^t \|\nabla u_{\notparallel}(\tau)\|_{L^2}^2 d\tau \le 8 \|u_{\notparallel}(s)\|_{L^2}^2$. 
\end{enumerate}
This is clearly a contradiction and hence $\widetilde{t}_0=\infty$. 

\bigskip
\appendix

\section{Energetic criterion for blow-up in the non-advective case} \label{app01}

We have already shown the global existence of the solution (in the sense of $L^2$) to the problem \eqref{maineq} under the assumptions when the flow $v$ has a small dissipation time or $v$ is a shear flow. However, the global existence is not always true for an arbitrary flow and the goal of this appendix is to include such an  example. Note that such examples have been well studied by El Soufi, Jazar and Monneau \cite{SJM07} (see, also \cite{JK08}), and here we would like to follow their presentation. 

Let us consider the easiest case when the flow $v \equiv 0$, that is:
\begin{equation} \label{e:nonadvmaineq}
\begin{cases}
u_t-\Delta u=|u|^p-\int_{\Omega} |u|^p \quad & \textrm{on} \quad \T^N, \\
\\
u \ \textrm{periodic} \quad & \textrm{on} \quad  \partial \Omega,
\end{cases}
\end{equation}
where $1<p \le 2$ and $u(x, 0)=u_0 \in C(\T^N)$ with mean zero. We have the following result. 

\begin{thm} {\cite[Theorem 1.1]{SJM07}} \label{L2blowup} 
Let
$$
E(u_0):=\int_{\T^N} \left(\frac{1}{2}\left|\nabla u_0\right|^2-\frac{1}{p+1}\left|u_0\right|^{p+1} \right) dx<0, 
$$
then the solution $u$ to the problem \eqref{e:nonadvmaineq} does not exist in $L^2(\T^N)$ for all $t>0$. Moreover, there exists $T>0$, such that if $u \in L_{loc}^\infty ([0, T]; L^2(\Omega))$, then 
$$
\lim_{t \to T^{-}} \|u(t)\|_{L^2(\T^N)}=+\infty. 
$$
\end{thm}

\begin{rem}
The problem \eqref{e:nonadvmaineq} was originally stated for any uniformly $C^2$ regular bounded domain in $\R^N$ with Neumann boundary conditions. Nevertheless, it is not hard to check that their proof still works under the periodic boundary conditions. 
\end{rem}

\section{A generalization of Theorem \ref{mainthm04}} \label{app02}
In the second appendix, we show that the assumptions in Theorem~\ref{mainthm04} can be relaxed. 
\begin{thm} \label{20210814thm01}
For $\alpha\geq 2$, $1<p<1+\frac{\alpha}{N+2}$, let $v \in L^\infty \left(\R_+; L^{\frac{\alpha}{p-1}} \right)$ and $u=\calN(u)$ be the mild solution on $[0, T]$ given in Theorem \ref{mainthm01} with the initial data $u_0 \in L^2(\T^N)$. Then, $u=\calN(u)$ is a weak solution of \eqref{maineq}, and satisfies the energy identity for any $0 \le t<T$:
\begin{eqnarray} \label{energyid_apen}
&& \|u_0\|_{L^2}^2+2\int_0^t\int_{\T^N} |u|^pu dxdt-2\int_0^t \left(\int_{\T^N} |u|^pdx \right)\left(\int_{\T^N} udx \right)dt \nonumber \\
&& \quad \quad \quad \quad \quad \quad  \quad \quad \quad =\|u(t)\|_{L^2}^2+2\int_0^t \int_{\T^N} \left|\nabla u \right|^2dxdt. 
\end{eqnarray}

\end{thm}
\begin{proof}
The proof remains the same as Theorem~~\ref{mainthm04} for the first three steps. While in \textit{Step IV}, we need to re-estimate the term $\int_\epsilon^t \int_{\T^N} \varphi(\tau) v(\tau) \cdot \nabla u(\tau) dxd\tau$ as follows. By using H\"older inequality with the triple $\left(\frac{2\alpha}{\alpha-2p+2},\frac{\alpha}{p-1},2\right)$, we have 
\begin{eqnarray*}
    \int_\epsilon^t \int_{\T^N} \varphi(\tau) v(\tau) \cdot \nabla u(\tau) dxd\tau %
    &\le &  \norm{\varphi}_{L^{\frac{2\alpha}{\alpha-2p+2}}([\epsilon,t]\times\T^{N})}\norm{v}_{L^{\frac{\alpha}{p-1}}([\epsilon,t]\times\T^{N})} \\
    && \cdot \norm{\nabla u}_{L^2([\epsilon,t]\times\T^{N})}
\end{eqnarray*}
While for the term $\norm{\varphi}_{L^{\frac{2\alpha}{\alpha-2p+2}}([\epsilon,t]\times\T^{N})}$, we use Gagliardo-Nirenberg inequality in dimension $N+1$ to get
\begin{equation*}
    \norm{\varphi}_{L^{\frac{2\alpha}{\alpha-2p+2}}([\epsilon,t]\times\T^{N})}\le C\left(\norm{D\varphi}_{L^2([\epsilon,t]\times\T^{N})}^{\frac{(N+1)(p-1)}{\alpha}}\norm{\varphi}_{L^2([\epsilon,t]\times\T^{N})}^{1-\frac{(N+1)(p-1)}{\alpha}}+\norm{\varphi}_{L^2([\epsilon,t]\times\T^{N})}\right).
\end{equation*}
Note that $0<\frac{(N+1)(p-1)}{\alpha}<1$, which is guaranteed by $1<p<1+\frac{\alpha}{N+2}$.

Next, we include some necessary modification in \textit{Step V}. By the Gagliardo-Nirenberg inequality on $\T^N$
\begin{equation*}
    \norm{\varphi}_{L^{\frac{2\alpha}{\alpha-2p+2}}(\T^{N})}\le C\left(\norm{D\varphi}_{L^2(\T^{N})}^{\frac{N(p-1)}{\alpha}}\norm{\varphi}_{L^2(\T^{N})}^{1-\frac{N(p-1)}{\alpha}}+\norm{\varphi}_{L^2(\T^{N})}\right),
\end{equation*}
 we can re-estimate the term $\int_{\epsilon}^t\int_{\T^N}\abs{\varphi_n-u}^{\frac{2\alpha}{\alpha-2p+2}}dxd\tau$ in the following way:
\begin{eqnarray} 
&& \int_{\epsilon}^t\int_{\T^N}\abs{\varphi_n-u}^{\frac{2\alpha}{\alpha-2p+2}}dxd\tau \nonumber=\int_{\epsilon}^t\norm{\varphi_n(\tau)-u(\tau)}_{L^{\frac{2\alpha}{\alpha-2p+2}}(\T^N)}^\frac{2\alpha}{\alpha-2p+2}d\tau \nonumber\\
&&\le\int_{\epsilon}^{t}\norm{\nabla\varphi(\tau)-\nabla u(\tau)}_{L^2(\T^N)}^{\frac{2(p-1)N}{\alpha-2p+2}}\norm{\varphi_n(\tau)-u(\tau)}_{L^2(\T^N)}^{\frac{2\alpha}{\alpha-2p+2}(1-\frac{(p-1)N}{\alpha})}d\tau \nonumber\\
&&\qquad+\int_{\epsilon}^t\norm{\varphi_n(\tau)-u(\tau)}_{L^2(\T^N)}^{\frac{2\alpha}{\alpha-2p+2}}d\tau\nonumber\\
&&\le\norm{\varphi_n-u}_{C([\epsilon,T];L^2)}^{\frac{2\alpha}{\alpha-2p+2}(1-\frac{(p-1)N}{\alpha})}\int_{\epsilon}^{t}\norm{\nabla\varphi_n(\tau)-\nabla u(\tau)}_{L^2(\T^N)}^{\frac{2(p-1)N}{\alpha-2p+2}}d\tau\nonumber\\
&&\qquad +(t-\epsilon)\norm{\varphi_n-u}_{C([\epsilon,T];L^2)}^{\frac{2\alpha}{\alpha-2p+2}}\nonumber\\
&&\le(t-\epsilon)^{1-\frac{(p-1)N}{\alpha-2p+2}}\norm{\varphi_n-u}_{C([\epsilon,T];L^2)}^{\frac{2\alpha}{\alpha-2p+2}(1-\frac{(p-1)N}{\alpha})}\norm{\nabla\varphi_n-\nabla u}_{L^2([\epsilon,T];L^2)}^{\frac{2(p-1)N}{\alpha-2p+2}}\nonumber\\
&&\qquad+(t-\epsilon)\norm{\varphi_n-u}_{C([\epsilon,T];L^2)}^{\frac{2\alpha}{\alpha-2p+2}}\nonumber
\end{eqnarray}
Note that the restriction $0<\frac{(p-1)N}{\alpha-2p+2}<1$ is guaranteed by $1<p<1+\frac{\alpha}{N+2}$.

Finally, in \textit{Step VI} we have
\begin{eqnarray*}
&&\int_\epsilon^t \int_{\T^N} |u(\tau)|^{p+1} dxdt\tau=\int_\epsilon^t \left\|u(\tau)\right\|_{p+1}^{p+1} d\tau \\
&&  \le   \int_\epsilon^t \left\|\nabla u(\tau) \right\|_{L^2}^{\frac{N(p-1)}{2}} \left\|u\right\|_{L^2}^{p+1-\frac{N(p-1)}{2}} d\tau +\int_{\epsilon}^{t}\|u\|_{L^2}^{p+1}d\tau \\
&&  \le  \|u\|_{X_T}^{p+1-\frac{N(p-1)}{2}}  \cdot \int_\epsilon^t \tau^{-\frac{N(p-1)}{4}} \cdot  \left(\tau^{\frac{1}{2}} \left\|\nabla u(\tau) \right\|_{L^2}\right)^{\frac{N(p-1)}{2}} d\tau + (t-\epsilon)\norm{u}_{X_T}^{p+1}\\
&& \le C_{T, p, N} \|u\|_{X_T}^{p+1},
\end{eqnarray*}
The other parts of the proof remain the same as Theorem~~\ref{mainthm04}.
\end{proof}


\begin{thebibliography}{ABCD99}

\bibitem{B77} Ball, J. M, Remarks on blow-up and nonexistence theorems for nonlinear evolution equations, \textit{The Quarterly Journal of Mathematics} {\bf28}(4) (1977): 473-486.

\bibitem{BBP19} Bedrossian, Jacob, Alex Blumenthal, and Samuel Punshon-Smith, Almost-sure exponential mixing of passive scalars by the stochastic Navier-Stokes equations, arXiv preprint arXiv:1905.03869 (2019).

\bibitem{BC17} Bedrossian, Jacob, and Michele Coti Zelati, Enhanced dissipation, hypoellipticity, and anomalous small noise inviscid limits in shear flows, \textit{Archive for Rational Mechanics and Analysis} {\bf224}(3) (2017): 1161-1204.

\bibitem{BH17} Bedrossian, Jacob, and Siming He, Suppression of blow-up in Patlak--Keller--Segel via shear flows, \textit{SIAM Journal on Mathematical Analysis} {\bf49}(6) (2017): 4722-4766.

\bibitem{CDE20}  M. Coti Zelati, M. G. Delgadino, and T. M. Elgindi, On the relation between enhanced dissipation timescales and mixing rates, \textit{Comm. Pure Appl. Math}. {\bf 73} (2020), no. 6, 1205–1244. 

\bibitem{CDFM21} M. Coti Zelati, Michele Dolce, Yuanyuan Feng, Anna L. Mazzucato, Global Existence for the Two-dimensional Kuramoto-Sivashinsky equation with a Shear Flow,
arXiv:2103.02971.

\bibitem{CKRZ08}  Constantin, Peter and Kiselev, Alexander and Ryzhik, Lenya and Zlato{\v{s}}, Andrej, Diffusion and mixing in fluid flow, \textit{Annals of Mathematics},(2008) 643--674

\bibitem{CWS94} Budd, Chris J., J. William Dold, and Andrew M. Stuart, Blow-up in a system of partial differential equations with conserved first integral. Part II: problems with convection, \textit{SIAM Journal on Applied Mathematics} {\bf54}(3) (1994): 610-640.

\bibitem{Fujita66} Fujita, Hiroshi, On the blowing up of solutions fo the Cauchy problem for $u_t= \Delta u+ u^{1+\alpha}$, \textit{J. Fac. Sci. Univ. Tokyo} {\bf13} (1966): 109-124.

\bibitem{FM20} Yuanyuan Feng, and Anna L. Mazzucato, Global existence for the two-dimensional Kuramoto-Sivashinsky equation with advection, arXiv preprint, arXiv:2009.04029 (2020).

\bibitem{FHX21A} Yu Feng, Bingyang Hu, and Xiaoqian Xu, Dissipation enhancement by mixing for evolution $ p $--Laplacian advection equations, arXiv preprint arXiv:2104.12578 (2021).

\bibitem{FHX21B} Yu Feng, Bingyang Hu, Xiaoqian Xu, Suppression of epitaxial thin film growth by mixing, arXiv preprint arXiv:2011.14088 (2021).

\bibitem{FFIT19} Yu Feng, Yuanyuan Feng, Gautam Iyer, and Jean-Luc Thiffeault. Phase separation in the advective Cahn-Hilliard equation. \textit{J. Nonlinear Sci}., {\bf 30}(6): 2821--2845, 2020.

\bibitem{FG19} Yuanyuan Feng and Gautam Iyer. Dissipation enhancement by mixing. \textit{Nonlinearity}, {\bf 32}(5):1810--1851, 2019.

\bibitem{Gallay18} Gallay, Thierry, Enhanced dissipation and axisymmetrization of two-dimensional viscous vortices, \textit{Archive for Rational Mechanics and Analysis} {\bf230}(3) (2018): 939-975.


\bibitem{GCM19} Alberti, Giovanni, Gianluca Crippa, and Anna Mazzucato, Exponential self-similar mixing by incompressible flows, \textit{Journal of the American Mathematical Society} {\bf32}(2) (2019): 445-490.

\bibitem{He18} Siming He, Suppression of blow-up in parabolic–parabolic Patlak–Keller–Segel via strictly monotone shear flows, \textit{Nonlinearity} {\bf31}(8) (2018): 3651.


\bibitem{IXZ21} Gautam Iyer, Xiaoqian Xu, and Andrej Zlatoš. Convection-induced singularity suppression
in the keller-segel and other non-linear pdes. \textit{Transactions of the American Mathematical
Society}, 2021. \text{https://doi.org/10.1090/tran/8195}. 

\bibitem{JK08} N. Jazar, R. Kiwan, Blow-up of a non-local semilinear parabolic equaiton with Neumann boundary conditions, \textit{Ann. I. H. Poincar\'e} -- AN {\bf 25} (2008), 215--218.

\bibitem{KX16} Kiselev, Alexander, and Xiaoqian Xu, Suppression of chemotactic explosion by mixing, \textit{Archive for Rational Mechanics and Analysis} {\bf222}(2) (2016): 1077-1112.

\bibitem{KSZ08}  Kiselev, Alexander, Roman Shterenberg, and Andrej Zlatoš, Relaxation enhancement by time-periodic flows, \textit{Indiana University mathematics journal} (2008): 2137-2152.

\bibitem{Levine90} Levine, Howard A, The role of critical exponents in blowup theorems, \textit{Siam Review} {\bf32}(2) (1990): 262-288.

\bibitem{LTC11} Zhi, Lin, Jean-Luc Thiffeault, and Stephen Childress, Stirring by squirmers, \textit{Journal of Fluid Mechanics} {\bf669} (2011): 167-177.

\bibitem{Pierrehumbert94} R. T. Pierrehumbert, Tracer microstructure in the large-eddy dominated regime, \textit{Chaos, Solitons \& Fractals} {\bf 4}(6):1091–1110, 1994

\bibitem{SJM07}  A. El Soufi, M. Jazar, R. Monneau, A Gamma-convergence argument for the blow-up of a non-local semilinear parabolic equation with Neumann boundary conditions, \textit{Ann. Inst. H. Poincar\'e Anal. Non Lin\'eaire} {\bf 24} (1) (2007) 17--39.

\bibitem{Thi12} Thiffeault, Jean-Luc, Using multiscale norms to quantify mixing and transport, \textit{Nonlinearity} {\bf25}(2) (2012): R1.

\bibitem{TZ19} Elgindi, Tarek M., and Andrej Zlatoš, Universal mixers in all dimensions, \textit{Advances in Mathematics} {\bf356} (2019): 106807.

\bibitem{Wei21} Wei, Dongyi. Diffusion and mixing in fluid flow via the resolvent estimate, \textit{Science China Mathematics} {\bf 64}(3) (2021): 507-518.

\bibitem{WZZ18} Wei, Dongyi, Zhifei Zhang, and Weiren Zhao, Linear inviscid damping for a class of monotone shear flow in Sobolev spaces, \textit{Communications on Pure and Applied Mathematics} {\bf71}(4) (2018): 617-687.

\bibitem{WZZ19} Wei, Dongyi, Zhifei Zhang, and Weiren Zhao, Linear inviscid damping and vorticity depletion for shear flows, \textit{Annals of PDE} {\bf5}(1) (2019): 1-101.

\bibitem{Zlatos10} Zlatoš, Andrej, Diffusion in fluid flow: dissipation enhancement by flows in 2D, \textit{Communications in Partial Differential Equations} {\bf35}(3) (2010): 496-534.

\bibitem{ZY17} Zlatoš, Andrej, and Yao Yao, Mixing and un-mixing by incompressible flows, \textit{Journal of the European Mathematical Society} {\bf19}(7) (2017): 1911-1948.



    

\end{thebibliography}
\end{document}